\documentclass[oneside,12pt]{amsart}
\usepackage[all]{xy}
\usepackage{amsmath}
\usepackage{amsfonts}

\usepackage{amscd,amsmath,latexsym, amssymb}
\usepackage{verbatim}
\usepackage{xcolor}
\usepackage{tikz}
\usepackage{dsfont}
\usetikzlibrary{matrix}
\begin{document}

\newtheorem{conjecture}{Conjecture}
\newtheorem{thm}{Theorem}[section]
\newtheorem{cor}[thm]{Corollary}
\newtheorem{lem}[thm]{Lemma}
\newtheorem{prop}[thm]{Proposition}

\newtheorem{defin}[thm]{Definition}
\theoremstyle{definition}
\newtheorem{exm}[thm]{Example}
\newtheorem{ques}{Question}
\newtheorem{remark}[thm]{Remark}
\theoremstyle{remark}
\newtheorem*{rem}{Remark}
\newtheorem*{ack}{Acknowledgment}

%% User definitions:

\newcommand{\A}{{\mathcal A}}
\newcommand{\R}{{\mathbb R}}
\newcommand{\C}{{\mathbb C}}
\newcommand{\Z}{{\mathbb Z}}
\renewcommand{\baselinestretch}{1.15}

%%Martin definitions
\def\zk{Z(K; (\underline{X},\underline{A}))}
\def\zkh{\widehat{Z}(K;\xa)}
\def\xa{(\underline{X},\underline{A})}
\def\gt{ Grbi\'{c} and  Theriault \hspace{1 \jot}}
%\def\qedalone{\hspace*{1in}\hfill ${\rm {\bf q.}{\bf e.}{\bf d.}}$}
%%%%%%%%%%%
\def\hz{H^*(Z(K; (\underline{X},\underline{A})))}
\def\hzI{H^*(\widehat{Z}(K_I; (\underline{X},\underline{A})))}\numberwithin{equation}{section}

\setlength{\oddsidemargin}{.25in}
\setlength{\textwidth}{6in}

%%End Martin definitions

\title[Spectral Sequence] {A spectral sequence for polyhedral products}

\author[A.~Bahri]{A.~Bahri}
\address{Department of Mathematics,
Rider University, Lawrenceville, NJ 08648, U.S.A.}
\email{bahri@rider.edu}

\author[M.~Bendersky]{M.~Bendersky}
\address{Department of Mathematics
CUNY,  East 695 Park Avenue New York, NY 10065, U.S.A.}
\email{mbenders@hunter.cuny.edu}

\author[F.~R.~Cohen]{F.~R.~Cohen}
\address{Department of Mathematics,
University of Rochester, Rochester, NY 14625, U.S.A.}
\email{cohf@math.rochester.edu}

\author[S.~Gitler]{S.~Gitler}

\subjclass[2010]{Primary: 55T05, 55T25, 55N10, 55N35,  \/ Secondary: 52C35}

%\keywords{arrangements, cohomology ring, moment-angle complex,
%simplicial complex, simplicial sets, stable splittings,
%Stanley-Reisner ring, suspensions, toric varieties}
\keywords{polyhedral products, cohomology, spectral sequences, stable splittings, moment-angle complexes,
simplicial complexes.}

\maketitle

{\bf This paper is dedicated to Samuel Gitler Hammer who brought us much joy and interest in Mathematics } 
\begin{abstract}

The purpose of this paper is to exhibit fine structure for polyhedral products $Z(K; (\underline{X},\underline{A}))$,
and polyhedral smash products $\widehat{Z}(K; (\underline{X},\underline{A}))$.
%\cite{bbcg, bbcg2}
 Moment-angle complexes are special cases for which $(X,A) = (D^2,S^1)$
 % \cite{buchstaber.panov}.
  There are three main parts to this paper.

\begin{enumerate}
\item One part  gives a natural filtration of the polyhedral product together with properties
of the resulting spectral sequence in Theorem \ref{thm:spectral sequence}. Applications of this
spectral sequence are given.

\item The second part uses the first to give a homological decomposition of $\widehat{Z}(K; (\underline{X},\underline{A}))$
CW pairs $(\underline{X},\underline{A})$. 
\item Applications to the ring structure of $Z(K; (\underline{X},\underline{A}))$
are given for  CW-pairs $(X,A)$ satisfying suitable freeness conditions.

\end{enumerate}

\

\end{abstract}

% \maketitle

%{\bf This paper is dedicated to Samuel Gitler Hammer who brought us much joy and interest in Mathematics } 

\section{\bf {Introduction}}\label{Introduction}

The subject of this paper is the homology of polyhedral products $Z(K; (\underline{X},\underline{A}))$, and polyhedral smash products $\widehat{Z}(K; (\underline{X},\underline{A}))$ \cite{davis.jan,buchstaber.panov,bbcg, bbcg2, bbcg3}.
Definitions are listed in section \ref{definitions} of this paper.

One of the purposes of this article is to give the Hilbert-Poincar\'e series for the polyhedral product
$Z(K; (\underline{X},\underline{A}))$ in terms of

\begin{enumerate}
\item the kernel, image, and cokernel of the induced maps $$H^*(X_i) \to H^*(A_i)$$ for all $i$, and

\item the full sub-complexes of $K$.

\end{enumerate}

\

This computation was also worked out in \cite{cartan}  using more geometric methods.

\

This is achieved by  analysis of a spectral sequence abutting to the cohomology of the
polyhedral product $Z(K; (\underline{X},\underline{A}))$ by filtering this space with the left-lexicographical ordering
of simplices.    The method applies to a generalized multiplicative cohomology theory, $h^*$ as well.  The spectral sequence is then used to describe some features of the ring structure of $h^*(\zk)$.

\

Qibing Zheng \cite{zheng} gives an alternative description of the cohomology of a polyhedral product.
Our methods are distinct from his and the presentation of the computational results assumes
a different form. Unlike the spectral sequence developed here, his collapses at the $E_2$ term.

\section{\bf {Definitions, and main results}}\label{definitions}

The basic constructions addressed in this article are defined in this section.
First recall the definition of an abstract simplicial complex.

\begin{defin}\label{defin: simplicial complex}

\begin{enumerate}

\item  Let $K$ denote an abstract simplicial complex with $m$ vertices labeled by the set $[m]=\{1,2,\ldots, m\}$. Thus, $K$ is a subset of the power set of $[m]$ such that
an element given by a $(k-1)$-simplex $\sigma$ of $K$ is given by an ordered sequence $\sigma = (i_1,\ldots, i_k)$ with $1 \leq i_1 <\cdots < i_k \leq m$ such that if $\tau \subset \sigma$, then
$\tau$ is a simplex of $K$. In particular the empty set $\emptyset$ is a subset of $\sigma$ and so it is in $K$. 
The vertex set of $\sigma$, $\{i_1,\ldots, i_k \}$ will be denoted $|\sigma|$.

\item Given a sequence  $I = (i_1, \ldots, i_k)$ with $1 \leq i_1 <\ldots <
i_k \leq m $, define $K_I \subseteq K$ to be the {\it full
sub-complex } of $K$ consisting of all simplices of $K$ which have
all of their vertices in $I$, that is $K_I = \{\sigma \cap I \ | \
\sigma \in K\}.$

\item  In case $I = (i_1, \ldots, i_k)$, define $X^I = X_{i_1} \times X_{i_2} \times \ldots \times X_{i_k}.$

\item Let $\Delta[m-1]$ denote the abstract simplicial complex given
by the power set of $[m]=\{1,2,\ldots, m\}$.

\end{enumerate}
\end{defin}

Let $h^*$ be a generalized, multiplicative cohomology theory and $(\underline{X},\underline{A})$ denote a collection of based CW pairs  $\{(X_i, A_i,x_i)\}^m_{i=1}$.  We will also assume $h^*(X_i)$ and $h^*(A_i)$ are finite type, i.e. $h^*(X_i)$ and $h^*(A_i)$ are  generated as  $h^*$ modules by classes, $\{x_{\ell}\}$ and $\{a_{\ell}\}$ respectively  with finitely many generators in each degree.

\

For a generalized cohomology theory, $h^*$, we now describe  a strong freeness condition on $\xa$  that will be imposed in Section \ref{sec:spectral sequence}.

\

The strong freeness condition assumes that  the long exact sequence

$$  \overset{\delta}{\to} \widetilde{h}^*(X_i/A_i) \overset{\ell}{\to} h^*(X_i) \overset{\iota}{\to} h^*(A_i)  \overset{\delta}{\to} \widetilde{h}^{*+1}(X_i/A_i) \to $$

\

can be written in terms of explicit, free $h^*$ modules $E_i, B_i, C_i$ and $W_i$.

\

\begin{defin} \label{def:bas}

The pair $\xa$ is said to satisfy a strong $h^*$ freeness condition if there are free $h^*$-modules $E_i, B_i, C_i$ and  $W_i$  satisfying

\begin{enumerate}
\item
$ h^*(A_i)= E_i\oplus B_i $   $ (B_i \ni 1 \subset  h^0(A_i)).$

\item $h^*(X_i)= B_i \oplus C_i$
\newline where $B_i \underset{\simeq}{\overset{\iota}{\to}} B_i, \quad C_i \overset{\iota}{\mapsto} 0$

\item $ \widetilde{h}^*(X_i/A_i)=C_i \oplus W_i$.
\newline where $C_i \underset{\simeq}{\overset{\ell}{\to}} C_i, \quad  E_i \underset{\simeq}{\overset{\delta}{\to}} W_i \overset{\ell}{\mapsto} 0$
\end{enumerate}

	\end{defin}      

\

The goal of the spectral sequence is to compute the cohomology of the polyhedral product defined below.  Our answer will be given in terms of the strong $h^*$ free decomposition described in Definition \ref{def:bas}.  In particular the description of the cohomology is only natural with respect to mappings of $h^*(X_i)$ and $h^*(A_i)$ which preserve the chosen strong $h^*$ decomposition.  This point is further developed at the end of section \ref{revisited}.

In the following definition $\mathcal{K}$ denotes the category of simplicial complexes and $\mathcal{CW}_* $ is the category of based CW pairs.

\begin{defin}\label{defin:gmac}

\begin{enumerate}
\item The {\it polyhedral product } determined by
$(\underline{X},\underline{A})$  and $K$ denoted
$$Z(K;(\underline{X},\underline{A}))$$ is defined using the functor $$D: \mathcal{K} \to
\mathcal{CW}_{\ast}$$ as follows: For every $\sigma$ in $K$, let
$$
D(\sigma) =\prod^m_{i=1}Y_i,\quad {\rm where}\quad
Y_i=\left\{\begin{array}{lcl}
X_i &{\rm if} & i\in \sigma\\
A_i &{\rm if} & i\in [m]-\sigma
\end{array}\right.$$ with $D(\emptyset) = A_1 \times \ldots \times A_k$.

\item The  polyhedral product is $$Z(K;(\underline{X},\underline{A}))=\bigcup_{\sigma \in K}
D(\sigma)= \mbox{colim} D(\sigma)$$ where the colimit is defined by
the inclusions, $d_{\sigma,\tau}$ with $\sigma \subset \tau$ and
$D(\sigma)$ is topologized as a subspace of the product $X_1 \times
\ldots \times X_k$. The {\it polyhedral product} is the underlying space
$Z(K;(\underline{X},\underline{A}))$ with base-point $\underline{*}
= (x_1, \ldots, x_k) \in Z(K;(\underline{X},\underline{A}))$.

\item In the
special case where $X_i = X$ and $A_i = A$ for all $1 \leq i \leq
m$, it is convenient to denote the polyhedral product
by $Z(K;(X,A))$ to coincide with the notation in
\cite{denham.suciu}.

\end{enumerate}
\end{defin}

A direct variation of the structure of the polyhedral product follows next.
Spaces analogous to polyhedral products are given next where products of spaces are
replaced by smash products, a setting in which non-degenerate  base-points are required.   We will always assume the pairs $(X,A)$ are based CW pairs, in which case the base point condition is always satisfied.

The (reduced) suspension of a (pointed) space $(X,*)$
$$\Sigma(X)$$ is the smash product $$S^1 \wedge X.$$

\begin{defin}\label{defin:smash.product.moment.angle.complex}
Given a polyhedral product
$Z(K;(\underline{X},\underline{A}))$ obtained from
$(\underline{X},\underline{A}, \underline{*})$, the {\it polyhedral
smash product}
$$\widehat{Z}(K;(\underline{X},\underline{A}))$$ is defined to be the image of
$Z(K;(\underline{X},\underline{A}))$ in the smash product $X_1
\wedge X_2 \wedge \ldots \wedge X_k$.

The image of $D(\sigma)$ in $\widehat{Z}(K;
(\underline{X},\underline{A}))$ is denoted by $\widehat{D}(\sigma)$
and is $$Y_1 \wedge Y_2 \wedge \ldots \wedge Y_k$$ where
$$Y_i=\left\{\begin{array}{lcl}
X_i &{\rm if} & i\in \sigma\\
A_i &{\rm if} & i\in [m]-\sigma.
\end{array}\right.$$

\

In case it is important to distinguish the pair $(\underline{X},\underline{A})$, the notations
$D(\sigma;(\underline{X},\underline{A}, \underline{*}))$, and $\widehat{D}(\sigma;(\underline{X},\underline{A}))$
will be used.

\end{defin}

As in the case of $Z(K;(\underline{X},\underline{A}))$, note that
$\widehat{Z}(K;(\underline{X},\underline{A}))$ is the colimit
obtained from the spaces $ \widehat{D}(\sigma) .$

\

\begin{defin}\label{defin:smash.products}
Consider an ordered sequence $I = (i_1, \ldots, i_k)$ with $1 \leq
i_1 <\ldots < i_k \leq m$ together with pointed spaces $Y_1, \ldots,
Y_m$. Then
\begin{enumerate}
\item the length of $I$ is $|I|= k$,
\item the notation $I \subseteq [m]$ means
$I$ is any increasing subsequence of $(1,\ldots, m)$,
\item $Y^{[m]}=Y_1 \times \ldots \times Y_m,$
\item $Y^{I} = Y_{i_1} \times Y_{i_2} \times \ldots \times
Y_{i_k},$
\item $\widehat{Y}^{I} = Y_{i_1} \wedge \ldots \wedge Y_{i_k},$
\end{enumerate}
\end{defin}
%%%%%%%%%%%%%%%%%%%%%%

Given a sequence  $I = (i_1, \ldots, i_k)$ with $1 \leq i_1 <\ldots < i_k \leq m $, define $K_I \subseteq K$ to be the {\it full
sub-complex } of $K$ consisting of all simplices of $K$ which have
all of their vertices in $I$, that is $K_I = \{\sigma \cap I \ | \
\sigma \in K\}.$ This notation is used for the first decomposition proven in \cite{bbcg, bbcg2} stated next.

\begin{thm} \label{thm:decompositions.for.general.moment.angle.complexes}
Let $K$ be an abstract simplicial complex with $m$ vertices. Given
$(\underline{X},\underline{A}) =\{(X_i, A_i, x_i)\}^m_{i=1}$ where
$(X_i,A_i,x_i)$ are pointed triples of CW-complexes
%for which \color{red}$X_i$, and $A_i$ are connected for all $i$,

there is  a natural pointed homotopy
equivalence
$$H: \Sigma(Z(K;(\underline{X},\underline{A})))\to \Sigma(\bigvee_{I \subseteq [m]}
\widehat{Z}(K_I;(\underline{X_I},\underline{A_I}))).$$
\end{thm}
%%%%%%%%%%%%%%%%%%%%%%

A second result in \cite{bbcg,bbcg2} is stated next where $|lk_{\sigma}(K)|$ denotes the geometric realization of the link of $\sigma$ in $K$.

\begin{thm}\label{thm:null.A}
Let $K$ be an abstract simplicial complex with $m$ vertices and
$\overline{K}$ its associated poset. Let
$(\underline{X},\underline{A})$  have the property that the
inclusion $A_i\subset X_i$ is null-homotopic for all $i$. Then there
is a homotopy equivalence
$$\widehat{Z}(K;(\underline{X},\underline{A}))\to\bigvee\limits_{\sigma\in
K} |\Delta(\overline{K}_{<\sigma})|*\widehat{D}(\sigma)$$ where
 $$|\Delta(\overline{K}_{<\sigma})| = |lk_{\sigma}(K)|$$  the link of $\sigma$ in $K$.

 \
 In particular if $X_i$ is contractible for all $i$ there is a homotopy equivalence

 $$\widehat{Z}(K;(\underline{X},\underline{A})) \to |K| \ast \widehat{A}^K.$$
Furthermore,  there is a
homotopy equivalence $$\Sigma (Z(K;(\underline{X},\underline{A})))
\to \Sigma(\bigvee\limits_{I\in [m]}(\bigvee\limits_{\sigma\in
K_I}|\Delta((\overline{K}_I)_{<\sigma})|*\widehat{D}(\sigma))).$$

\end{thm}

\

Theorem \ref{thm:null.A} for the case $X_i$ contractible for all $i$ is called the {\it wedge lemma}.  

\

A filtration on $\zk$ is next described. The purpose of introducing this filtration is that there is an associated spectral sequence which is the subject of the article.  The spectral sequence  converges to the cohomology of $\zk$.  

\begin{defin}\label{defin: lexicographical ordering}

 The $(m-1)$-simplex $\Delta[m-1]$ is totally ordered by the left-lexicographical ordering of all faces defined as follows:
 $$\sigma = (i_1,i_2,..., i_s) < \tau = (j_1,j_2,..., j_t)$$ if and only if either

\begin{enumerate}
\item $ 1 \leq s < t \leq m$ or
\item $t = s$, and there exists an integer $n$ such that $(i_1,i_2,..., i_n) = (j_1,j_2,..., j_n)$ but \\
$i_{n+1} < j_{n+1}$.
\end{enumerate}

There are $$(1+   \binom{m}{1}  + \binom{m}{ 2}+ \binom{m}{ 3} + \cdots + \binom {m} { m-1} +\binom {m}{m}) = 2^m$$ faces in $\Delta[m-1]$ (including the empty set) which are totally ordered by the integers $q$ such that $0 \leq q \leq 2^m-1$.

\

Furthermore, let $\sigma_0$ denote the emptyset $\emptyset$; thus $\sigma_0 \leq \sigma $ for all $\sigma$ in $\Delta[m-1]$.

The weight of a face $\sigma$ is that integer $q,$ denoted by $wt(\sigma)$  where $q$ is the position of $\sigma$ in this total left lexicographical ordering of the simplices.

The $(m-1)$-simplex $\Delta[m-1]$ is filtered by requiring  $$F_t\Delta[m-1] = \cup_{wt(\sigma) \leq t}\sigma. $$
\end{defin}

This filtration of $\Delta[m-1]$ induces a filtration of $K$ as given next.

 \begin{defin}\label{defin: filtration of polyhedral products}
 The $(m-1)$-simplex $\Delta[m-1]$ is filtered by the left lexicographical ordering of all faces as in Definition
 \ref{defin: lexicographical ordering}. Let $K$ be a simplicial complex with $m$ vertices.

Filter $K$ by $$F_tK = K \cap F_t\Delta[m-1].$$

Filter the polyhedral product $Z(K;X,A))$ and polyhedral smash product
$\widehat{Z}(K;(X,A))$  by

\begin{enumerate}
\item $$F_tZ(K;(X,A)) = \cup_{wt(\sigma) \leq t} D(\sigma;(X,A)),$$ and
\item  $$F_t\widehat{Z}(K;(X,A)) = \cup_{wt(\sigma) \leq t} \widehat{D}(\sigma;(X,A)).$$
\end{enumerate}

\end{defin}

\

Record this information stated as the next lemma.

\

\begin{lem} \label{lem: filtration of a simplicial complex}
There is a total ordering of all of the faces of a simplicial complex $K$ given by
\begin{enumerate}
\item the left-lexicographical ordering of all of the faces of $\Delta[m-1]$, and
\item the induced ordering via the natural inclusion $$K \subset \Delta[m-1].$$
\end{enumerate}

Furthermore, inclusions $$L \subset K$$ induced by an embedding of simplicial
complexes with $m$ vertices is order preserving, and filtration preserving where
$F_tK = K \cap F_t\Delta[m-1]$ as listed in Definition \ref{defin: filtration of polyhedral products}.
Namely, the inclusion $L \subset K$ induced by an embedding of simplicial
complexes with $m$ vertices is a morphism of filtered complexes $$F_*L \subset F_*K.$$

This filtration of $K$ induces a filtration of the polyhedral product $Z(K;X,A))$ and polyhedral smash product
$\widehat{Z}(K;(X,A))$ given by
\begin{enumerate}
\item $$F_tZ(K;(X,A)) = \cup_{wt(\sigma) \leq t} D(\sigma;(X,A)),$$ and
\item  $$F_t\widehat{Z}(K;(X,A)) = \cup_{wt(\sigma) \leq t} \widehat{D}(\sigma;(X,A)).$$
\end{enumerate} Furthermore, the natural quotient map
$Z(K;X,A)) \to \widehat{Z}(K;(X,A))$ is filtration preserving.

\end{lem}

\

\begin{remark}  The filtration constructed in Lemma \ref{lem: filtration of a simplicial complex} is exploited in \cite{cartan}.
\end{remark}

 \begin{defin}\label{defin: notation for spectral sequence}
 
 \
 
\

\begin{enumerate}
\item If $\sigma \in K$, write $$(\underline{X}/\underline{A})^{\sigma}$$ for the  smash product
 $(X_{i_1}/A_{i_i})\wedge \cdots \wedge (X_{i_q}/A_{i_q})$ where   $I=(i_1, \cdots , i_q)$  is as in Definition \ref{defin:smash.products} and $\sigma$ has vertex set $I$.

 \item Write $$\underline{A}^{\sigma^c}$$ for the product
 $ A_{j_1} \times \cdots \times A_{j_{k-q}}   $ where $\sigma \cup \{j_{1}, \cdots, j_{k-q}\}=[m]$, and $\sigma^c $ denotes the complement of $\sigma$. In particular for $\sigma =\emptyset, A^{\sigma^c}=A_1 \times \cdots \times A_k$.

%\item  For $X$ a pointed space, and $Y$ a space, the right half smash product, $X \rtimes Y$ is defined to be $X \wedge Y_+$.
% \item  Choose  an ordering of the vertices of  $K$.  The induced lexicographical ordering of the simplices of $K$ is defined by $\sigma < \tau$ if the dimension of $\sigma$ is less than the dimension of $\tau$ and  $$(i_1, \cdots, i_t, \cdots i_n )< (j_1, \cdots, j_t, \cdots j_n)$$ if $i_1 = j_1, \cdots, i_{t-1} =j_{t-1}, i_t < j_t$.

\end{enumerate}
 \end{defin}

Half-smash products are basic in this setting with their definition as follows.

 \begin{defin}\label{defin: half smash products}
 Let $$(X, x_0) \  \mbox{and} \ (Y,y_0)$$ denote pointed spaces.
 Define $$X\rtimes Y = (X \times Y)/(x_0 \times Y),$$
 and $$X \ltimes Y = (X \times Y)/ (X \times y_0).$$
\end{defin}

\
An example is given next.

\begin{exm} \label{exm:filtrations}
Let $K$ denote the simplicial complex with two vertices
$\{1,2\}$ and with one edge $(1,2)$.  Then $Z(K; (\underline{X},\underline{A})) = X_1 \times X_2.$
The filtration of $X_1 \times X_2$ given in Definition \ref{defin: lexicographical ordering}
 is stated next.

 \begin{enumerate}
\item $F_0Z(K; (\underline{X},\underline{A})) = A_1 \times A_2$,
\item $F_1Z(K; (\underline{X},\underline{A})) = X_1 \times A_2$,
\item $F_2Z(K; (\underline{X},\underline{A})) = (X_1 \times A_2) \cup (A_1\times X_2)$, and
\item $F_3Z(K; (\underline{X},\underline{A})) = X_1 \times X_2$.
\end{enumerate}

Let $$F_i = F_iZ(K; (\underline{X},\underline{A}))$$ in this example. If $(X_i,A_i)$ are pairs of finite CW-complexes, there are homeomorphisms

 \begin{enumerate}
\item $F_1/F_0 \to (X_1/A_1 \times A_2)/ (* \times A_2) \to X_1/A_1 \rtimes A_2$,
\item $F_2/F_1 = (X_1 \times A_2) \cup (A_1\times X_2)/ (X_1 \times A_2 ) \to
A_1 \ltimes( X_2/A_2),$ and
\item $F_3/F_2 = X_1 \times X_2/(X_1\times A_2 \cup A_1\times X_2 \to (X_1/A_1)\wedge (X_2/A_2)$.
\end{enumerate}

Letting $[x]$ denote image of  the projection of $x \in X_1$ to $X_1/A_1$ then  the homeomorphism in (2) is given by 
$$(X_1 \times A_2) \cup (A_1 \times X_2)/(X_1 \times A_2) \simeq (A_1 \times X_2)/(A_1 \times A_2) \overset{p}{\to} A_1 \ltimes( X_2/A_2)$$ with $p(a \times x)=a \times [x]$

The homeomorphism in (3) is a special case of Lemma \ref{lem: cofibrations}

\

The filtrations and their associated graded for the smash polyhedral products are exhibited next.

 \begin{enumerate}
\item $F_0\widehat{Z}(K; (\underline{X},\underline{A})) = A_1 \wedge A_2$,
\item $F_1\widehat{Z}(K; (\underline{X},\underline{A})) = X_1 \wedge A_2$,
\item $F_2\widehat{Z}(K; (\underline{X},\underline{A})) = (X_1 \wedge A_2) \cup (A_1\wedge X_2)$, and
\item $F_3\widehat{Z}(K; (\underline{X},\underline{A})) = X_1 \wedge X_2$.
\end{enumerate}

Let $$F_i = F_i\widehat{Z}(K; (\underline{X},\underline{A}))$$ in this example where $(X_i,A_i)$ are assumed to be pairs of finite CW-complexes. There are homeomorphisms

 \begin{enumerate}
\item $F_1/F_0 \to (X_1/A_1) \wedge A_2,$
\item $F_2/F_1 = (X_1 \wedge A_2) \cup (A_1\wedge X_2)/ X_1 \wedge A_2  \to A_1 \wedge (X_2/A_2),$ and
\item $F_3/F_2 = X_1 \times X_2/(X_1\times A_2 \cup A_1\times X_2 \to (X_1/A_1)\wedge (X_2/A_2)$.
\end{enumerate}

\end{exm}

Given a filtered space, there is a natural spectral sequence associated to that filtration.
The next theorem records the  properties of the resulting spectral sequence of a filtered space
in the context of polyhedral products with the left-lexicographical ordering obtained from
Definition \ref{defin: lexicographical ordering}.

\

\begin{thm}\label{thm:spectral sequence}
The left-lexicographical ordering of simplices induces spectral sequences of a filtered spaces 
\begin{enumerate}
\item $E_r(K; \xa) \Rightarrow \widetilde{h}^*(\zk) $ with  $$E_1(K; \xa) = \underset{\sigma \in K}{\bigoplus} \widetilde{h}^*(  (\underline{X}/\underline{A})^{\sigma} \rtimes \underline{A}^{\sigma^c} ),$$ 

 and a spectral sequence
\item 
$E_r(\widehat{K}; \xa) \Rightarrow \widetilde{h}^*(\zkh) $

 $$E_1^{s,t}(\widehat{Z}(K; (\underline{X}, \underline{A}))) = \underset{\sigma \in K}{\bigoplus} \widetilde{h}^*( (\underline{X}/\underline{A})^{\sigma}) \wedge \widehat{\underline{A}}^{\sigma^c} ).$$ \end{enumerate}

The notation defined in \ref{defin: notation for spectral sequence}.   The grading, $s$ is the index of the simplex, $\sigma$ in the left-lexicographical ordering.  $t$ is the cohomological degree.

The differentials  satisfy $$d_r: E_r^{s,t} \to E_r^{s+r,t+1}.$$

Furthermore, the spectral sequence is natural for embeddings of simplicial maps, $L \subset K$ with the same number of vertices and with respect to maps of pointed pairs $\xa \to (\underline{Y},\underline{B})$.

The natural quotient map $$Z(K; (\underline{X}, \underline{A})) \to \widehat{Z}(K; (\underline{X}, \underline{A}))$$ induces a morphism of spectral sequences,  and the stable decomposition of Theorem \ref{thm:decompositions.for.general.moment.angle.complexes}
induces a morphism of spectral sequences.

  \end{thm}

 \
 
 \begin{remark}
 We remark  that $\zk$ and the spectral sequence commutes with colimits in $\xa$.  

\

 We also note that we only use the fact that the inclusions $F_t \subset F_{t+1}$ are cofibrations.  This follows from the hypothesis that that $(X_i,A_i)$ are finite CW pairs.  The argument generalizes to NDR pairs. 
 \end{remark}

Some consequences  of this spectral sequence are worked out below.  An explicit description of the cohomology of $Z(K; (\underline{X},\underline{A}))$ with field coefficients $\mathbb F$ will be given next followed by a section on examples. The answers for cohomology are given in terms of kernels and cokernels
of $$H^i(X_j) \to H^i(A_j).$$
\
%\color{blue} your definition of $SR(X,A)$ should include the complex, $K$. \color{red}
\begin{defin}\label{defin:SR.ideals.again}

Assume that $K$ is a simplicial complex and the pointed pairs $(\underline{X},\underline{A},\underline{*})$ are of finite type.  Assume that
the maps $A_i \to X_i$ induce split surjections 
in cohomology with field coefficients
$\mathbb F$. Consider the kernel of $H^i(X_j) \to H^i(A_j)$ together with the elements
$x_j \in \mbox{kernel}(H^i(X_j) \to H^i(A_j))$ together with the two-sided ideal
generated by all such $x_{i_1}\otimes  x_{i_2} \otimes \cdots \otimes x_{i_t}$ with $(i_1, i_2, \cdots, i_t)$ not a simplex in $K$, denoted
$$SR(K; (\underline{X},\underline{A})).$$

\end{defin}
\color{black}

%\begin{defin}\label{defin:SR.ideals.again}

%Assume that the pointed pairs $(\underline{X},\underline{A},\underline{*})$ are of finite type, and assume that the maps $A_i \to X_i$ induce split monomorphisms in homology with field coefficients $\mathbb F$. Consider the kernel of $H^i(X_j) \to H^i(A_j)$ together with the elements $x_j \in \mbox{kernel}(H^i(X_j) \to H^i(A_j))$ together with the two-sided ideal  generated by all such $x_1\otimes x_2 \otimes \cdots \otimes x_n$ denoted  $$SR(\underline{X},\underline{A}).$$ \end{defin}

A result which is analogous to Theorem $2.35$ of \cite{bbcg} follows next.

\begin{thm} \label{thm:Cartan_and_split.monomorphisms}
Let $K$ be an abstract simplicial complex with $m$ vertices. Assume
that $$(\underline{X},\underline{A},\underline{*})$$ are pointed triples of connected
CW-complexes of finite type for all $i$ for which cohomology is taken with field coefficients $\mathbb F$.
If the maps $A_i \to X_i$ induce split surjections in cohomology, then the induced map $$H^*(X^{[m]}) \to H^*(Z(K; (\underline{X},\underline{A})))$$ is an epimorphism of algebras which is additively split. Furthermore, there is an induced isomorphism of algebras $$H^*(X^{[m]})/SR(K; (\underline{X},\underline{A})) \to H^*(Z(K; (\underline{X},\underline{A}))).$$

\end{thm}

\

Recall that a map $A_i \to X_i$ induces a split monomorphism in integer homology if and only
if it induces a split monomorphism with field coefficients for every prime field $\mathbb F_p$
and the rational numbers. A corollary of Theorem  \ref{thm:Cartan_and_split.monomorphisms}
which follows immediately is stated next. 

\

\begin{cor} \label{thm:Cartan_and_split.monomorphisms.INTEGRALLY}
Let $K$ be an abstract simplicial complex with $m$ vertices. Assume
that $$(\underline{X},\underline{A}, \underline{*})$$ are pointed triples of connected
CW-complexes of finite type for all $i$ for which cohomology is taken with coefficients $\mathbb Z$.
Assume that the maps $A_i \to X_i$ induce split monomorphisms in homology over $\mathbb Z$, then the induced map $$H^*(X^{[m]};\mathbb Z) \to H^*(Z(K; (\underline{X},\underline{A})); \mathbb Z)$$ is an epimorphism of algebras
which is additively split.
\end{cor}

\begin{remark}
	The proof of Theorem \ref{thm:Cartan_and_split.monomorphisms} works just as well for any multiplicative cohomology $h^*$ and CW pairs with $h^*(X_i)$ and $h^*(A_i)$ finitely generated free $h^*$ modules. 
\end{remark}

%\begin{thm} \label{thm:wedge.homology.decompositions}
%Let $K$ be an abstract simplicial complex with $m$ vertices. Assume
%that $(\underline{X},\underline{A})$ are pointed triples of
%CW-complexes for all $i$. The homology of $Z(K;
%(\underline{X},\underline{A}))$ with any field coefficients $\mathbb F$ depends only
%on $K$, and  $$H_*(A_i; \mathbb F) \to H_*(X_i; \mathbb F)$$ for all $i$.
%\end{thm}

\

 \

%%%%%%%%%%%%%%%%%%%%%%%%%%%%%%%%%%%%%%%%%%%%%%%%%%%%%%%
\section{\bf A spectral sequence for the cohomology of $Z(K,\xa)$ and $\zkh$}\label{sec:spectral sequence one}

\

 The object of this section is to construct the  spectral sequences of Theorem \ref{thm:spectral sequence}.  In subsequent sections these spectral sequences will be used to compute the cohomology of $Z(K; \xa)$ when $\xa$ satisfies suitable flatness conditions.
 %{\bf Before giving details, it should be pointed out that this spectral sequence is a special case of the spectral arising from a filtered space with a filtration given for the polyhedral product. Furthermore, this filtration is induced by a refinement of G.~Whitehead's classical filtration of a product.}

 \

The spectral sequences of Theorem \ref{thm:spectral sequence} are precisely those obtained by filtering the spaces
$Z(K,\xa)$ and $\widehat{Z}(K; (\underline{X}, \underline{A}))$ by finite filtrations induced by the left-lexicographical ordering. Since these spectral sequences arise by finite filtrations,  the spectral sequences converge in the strong sense. It remains to identify the associated graded $E_0$ as well as $E_1$, and the first differential.

% \begin{enumerate}
%\item $E_r(K; \xa) \Rightarrow \widetilde{H}^*(\zk) $ with  $$E_1(K; \xa) = \underset{\sigma \in K}{\bigoplus} \widetilde{H}^*(  (\underline{X}/\underline{A})^{\sigma} \rtimes \underline{A}^{\sigma^c} ),$$ as well as, and \item $$E_1(\widehat{Z}(K; (\underline{X}, \underline{A}))) = \underset{\sigma \in K}{\bigoplus} \widetilde{H}^*( \widehat{D}(\sigma;(\underline{X},\underline {A})) \wedge \widehat{\underline{A}}^{\sigma^c} ).$$ \end{enumerate}

%\

%\begin{thm}\label{thm:th1}
% $$E_r(K,\xa) \Rightarrow \widetilde{H}^*(\zk) $$ with
%$$E_1(K,\xa) = \underset{\sigma \in K}{\bigoplus} \widetilde{H}^*(  (\underline{X}/\underline{A})^{\sigma} \rtimes \underline{A}^{\sigma^c} ) $$

\

The next three lemmas  give the identification of $E_0$.

\

Suppose that $(X, x_0) \  \mbox{and} \ (Y,y_0)$ are both pointed CW complex, then the right and left half-smash products were defined in Definition \ref{defin: half smash products} by $$X\rtimes Y = (X \times Y)/(x_0 \times Y), \ \mbox{and} \
X \ltimes Y = (X \times Y)/ (X \times y_0).$$  A useful lemma follows in which $X_{+}$ denotes $X$ with a disjoint base-point added.

\

 \begin{lem} \label{lem: cofibrations.two}

Let $$(X, x_0), \   \mbox{and} \  (Y,y_0)$$ be  pointed, finite CW pairs. Then there are homotopy equivalences

\begin{enumerate}
\item  $\Sigma(X \rtimes Y ) \to \Sigma(X \wedge Y) \vee \Sigma(X)$,
\item  $\Sigma(X \ltimes Y ) \to \Sigma(X \wedge Y) \vee \Sigma(Y)$, and
\item  $X \ltimes Y = (X \times Y)/ (X \times y_0) \to  X_{+} \wedge Y.$
%(X_{+} \times Y)/ (X \times y_0 \cup   \{+\} \times Y)= X_{+} \wedge Y.$$ a homeomorphism.

\end{enumerate}

 \end{lem}

%%%%%%%%%%%%%%%%%
\begin{lem} \label{lem: cofibrations}

Let $$(X_i, A_i)$$ be  finite CW pairs. Let $\Lambda$ denote the subspace of
$$X_1 \times X_2 \times \cdots \times X_n$$ given by
$$\Lambda = \cup_{1 \leq i \leq m} X_1 \times X_2 \times \cdots  \times X_{i-1} \times A_i \times X_{i+1} \times \cdots \times X_m.$$ There is a  natural  homeomorphism $$\theta: X_1 \times X_2 \times \cdots X_m/\Lambda \to (X_1/A_1) \wedge (X_2/A_2) \wedge \cdots \wedge (X_m/A_m).$$
\end{lem}

\

\begin{proof} Let  $[x_i]$ denote the image of $x_i \in X_i$ of   the projection $X_i \to X_i/A_i.$  The natural  map 
$$\phi: X_1 \times X_2 \times \cdots X_m\to (X_1/A_1) \wedge (X_2/A_2) \wedge \cdots \wedge (X_m/A_m)$$
which send $(x_1,\cdots,x_m) $ to $([x_1], \cdots, [x_m])$
is a continuous surjection.  The  class $(x_1,\cdots,x_m) $ maps to the base point in  $(X_1/A_1) \wedge (X_2/A_2) \wedge \cdots \wedge (X_m/A_m)$ if and only if at least one of the factors, $x_i$ maps to the base point in $X_i/A_i$.  Equivalently at least one of the factors $ x_i \in A_i$.  In particular $\phi$ factors through a map, $\theta$.  By construction $\theta$ is a bijection. The lemmas follow for compact CW complexes with $A_i$ closed in $X_i$, since the target space is Hausdorff while the domain is compact. Thus the natural map is a homeomorphism.

\

\end{proof}

\begin{remark}
	  The lemmas extends to locally finite  CW complexes  by taking a limit over finite skeleta. 
\end{remark}

%Suppose that $(X, x_0) \  \mbox{and} \ (Y,y_0)$ are both pointed CW complex, then the right and left half-smash products were defined in Definition \ref{defin: half smash products} by $$X\rtimes Y = (X \times Y)/(x_0 \times Y), \ \mbox{and} \  X \ltimes Y = (X \times Y)/ (X \times y_0).$$
Another useful lemma follows.
% in which $X_{+}$ denotes $X$ with a disjoint base-point added.

 %\begin{lem} \label{lem: cofibrations.two}
% Let $$(X, x_0), \   \mbox{and} \  (Y,y_0)$$ be  pointed, finite CW pairs. Then there are homotopy equivalences\begin{enumerate} \item  $\Sigma(X \rtimes Y ) \to \Sigma(X \wedge Y) \wedge \Sigma(X)$,\item  $\Sigma(X \ltimes Y ) \to \Sigma(X \wedge Y) \wedge \Sigma(Y)$, and \item  $X \ltimes Y = (X \times Y)/ (X \times y_0) \to  X_{+} \wedge Y.$
%(X_{+} \times Y)/ (X \times y_0 \cup   \{+\} \times Y)= X_{+} \wedge Y.$$ a homeomorphism. \end{enumerate} \end{lem}

%%%%%%%%%%%%%%%%%
\begin{lem} \label{lem: half.smashcofibrations}

Let $$(Y_i, A_i)$$ be finite, pointed CW pairs.  Then there is a homeomorphism
$$(Y_1 \times Y_2 )/(A_1 \times Y_2) \to (Y_1/A_1) \rtimes (Y_2).$$ 

\

Thus there are homeomorphisms

\begin{multline*}
(Y_1 \times Y_2 \times \cdots \times Y_n )/(A_1 \times Y_2 \times \cdots \times Y_n ) \to \\
(Y_1/A_1) \rtimes (Y_2 \times \cdots \times Y_n ) \to (\cdots (Y_1/A_1) \rtimes (Y_2)) \rtimes Y_3) \cdots \rtimes Y_n ).
\end{multline*}

\end{lem}

\

\begin{proof}
The natural quotient map is a continuous bijection by inspection. Since all spaces are finite complexes with
$A_i$ closed in $Y_i$, the target space is Hausdorff while the domain is compact. Thus the natural map is a homeomorphism.

\end{proof}

%%\begin{equation}\label{rtimesProperty} \Sigma(X \rtimes A ) \to \Sigma(X \wedge A) \wedge \Sigma(X). \end{equation}
%Since both $Z(K;X,A))$, and  $\widehat{Z}(K;(X,A))$ are filtered and the natural quotient map $Z(K;X,A)) \to \widehat{Z}(K;(X,A))$ preserves filtration by Lemma \ref{lem: filtration of a simplicial complex}, there are the associated spectral sequences for filtered spaces together with an identification of the $E^1$-term obtained directly from Lemma \ref{lem: cofibrations}.

\

In the next lemma, abbreviate $F_sZ(K;(X,A))$ by $F_sZ$, and $F_s\widehat{Z}(K;(X,A))$ by
$F_s\widehat{Z}$.  Let $(X_i,A_i), i= 1,\cdots, m$ be finite, pointed CW pairs.   The filtrations of Definition \ref{defin: filtration of polyhedral products}  are given by
\begin{enumerate}
\item $$F_tZ(K;(X,A)) = F_tZ = \cup_{wt(\sigma) \leq t} D(\sigma;(X,A)),$$ and
\item  $$F_t\widehat{Z}(K;(X,A)) = F_t\widehat{Z} = \cup_{wt(\sigma) \leq t} \widehat{D}(\sigma;(X,A))$$
\end{enumerate} have the property that  the natural quotient map
$$Z(K;X,A)) \to \widehat{Z}(K;(X,A))$$ is filtration preserving, and are natural for morphisms of simplicial
complexes $$L \to K$$ which induce an isomorphism of sets on the vertices.

\

The next step is to identify the filtration quotients $F_sZ/ F_{s-1}Z$ as well as $F_s\widehat{Z}/F_{s-1}\widehat{Z}.$
 \begin{lem} \label{lem:associated graded}

Let $(X_i,A_i), i= 1,\cdots, m$ be finite, pointed CW pairs.   The filtrations of Definition \ref{defin: filtration of polyhedral products}  have the following property. If $\sigma$ is the maximal simplex occurring in 
$$F_tZ(K;(X,A)) = F_tZ = \cup_{wt(\sigma) \leq t} D(\sigma;(X,A)),$$ 
then the natural quotient maps $$F_tZ/ F_{t-1}Z \to
 (\underline{X}/\underline{A})^{\sigma} \rtimes \underline{A}^{\sigma^c}\; \; \text{and}\;\;\;
F_t\widehat{Z}/F_{t-1}\widehat{Z} \to
 (\underline{X}/\underline{A})^{\sigma} \wedge \widehat{\underline{A}^{\sigma^c}}$$ 
 are homeomorphisms.

 \end{lem}

\

 \begin{proof}

Suppose $F_sK$ is obtained from $F_{s-1}K$ by attaching an $n$-simplex $\Delta$.  For simplicity we may, after relabelling,  assume $\Delta$ has vertices $\{1,2 \cdots, n+1\}$.

$\Delta$ and its boundary, $\partial \Delta$ can be viewed as simplicial complexes on the vertex set $[m]$.  The vertices $\{n+2, \cdots, m\}$ are not zero simplices.  In the literature  they are referred to as ghost vertices.

It follows from \gt \cite{GT} that there is a commutative  diagram of cofibrations

\begin{equation}\label{diag:pushout}
 \begin{array}{ccccc}
Z(\partial \Delta; \xa)&\overset{j}{ \to} & Z(\Delta; \xa) & \to &C\\
\downarrow&&\downarrow&& \Vert \\
F_{s-1}Z & \to & F_sZ & \to &C
\end{array}
\end{equation}

Now  observe that

\begin{enumerate}

\item $Z(\Delta;\xa) = X_1 \times \cdots \times X_{n+1} \times A_{n+2} \times \cdots \times A_m$.

\item $Z(\partial \Delta; \xa ) = Z(\partial \overline{\Delta} ; \xa) \times  A_{n+2} \times \cdots \times A_m$
where $\partial \overline{\Delta}$ is a simplicial complex on $[n+1]$.  While $\partial \Delta$ is a simplicial complex on $[m]$ with the  set of $0-$simplices $=\{1,\cdots,n+1\}$ and ghost vertices $=\{n+2, \cdots m\}$.

\item If $\partial \overline{\Delta} \subset [n+1]$ then

$$Z(\partial \overline{\Delta}; \xa) = \underset{q}{\bigcup} X_1 \times \cdots \times \A_q \times \cdots X_{n+1}.$$

\item By lemma \ref{lem: cofibrations}, $$X_1 \times \cdots \times X_{n+1}/Z(\partial \overline{\Delta}; \xa) = (X_1/A_1) \wedge \cdots \wedge (X_{n+1}/A_{n+1}).$$

\item  It follows from Lemma \ref{lem: half.smashcofibrations} that for $A \subset X$ a CW pair the cofiber of $A  \times Y \hookrightarrow X \times Y$ is the right half smash $(X/A) \rtimes Y$.

\end{enumerate}

The following description of $C$ follows from the top row of diagram \ref{diag:pushout} and these  observations

\bigskip
 $$C = F_sZ/F_{s-1}Z= (X_1/A_1) \wedge \cdots \wedge (X_{n+1}/A_{n+1}) \rtimes  (A_{n+2} \times \cdots \times A_m).$$

 A similar argument using part (2) of Lemma \ref{lem: cofibrations} shows that

\begin{lem} \label{lemma:relhatspace}
 $$\widehat{F}_sZ/\widehat{F}_{s-1}Z = (X_1/A_1) \wedge \cdots \wedge (X_{n+1}/A_{n+1}) \wedge  (A_{n+2} \wedge \cdots \wedge A_m) = \widehat{D}(\sigma; (\underline{X/A},\underline{A}).$$
\end{lem}
\vspace {3 \jot}

This completes the description of the $E_1$ page and the proof of \ref{thm:spectral sequence}.

 \begin{comment}
 There are two natural cases to check according to the left-lexicographical ordering as follows.
Assume that $\sigma$ is the maximal simplex occurring in $F_tZ(K;(X,A)) = F_tZ = \cup_{wt(\sigma) \leq t} D(\sigma;(X,A))$ with respect to the left-lexicographical ordering of simplices. Assume that $\tau$ is the penultimate simplex
so that

\begin{enumerate}
\item $\tau < \sigma$, and
\item if $\tau \leq \lambda \leq \sigma$, then either $\tau = \lambda$, or $\lambda = \sigma$.
\end{enumerate}

In either case, the lemma follows directly from

Notice that there is a natural projection $F_tZ \to  (\underline{X}/\underline{A})^{\sigma} \rtimes \underline{A}^{\sigma^c} $ which passes to a continuous bijection
$$F_tZ/F_{t-1}Z \to  (\underline{X}/\underline{A})^{\sigma} \rtimes \underline{A}^{\sigma^c}.$$ It suffices to check that
this map is a homeomorphism, at least in special cases.

With the hypotheses that $(X_i, A_i)$ are pointed pairs of finite CW-complexes, it follows that $F_tZ/F_{t-1}Z$ is compact, and that $(\underline{X}/\underline{A})^{\sigma} \rtimes \underline{A}^{\sigma^c}$ is Hausdorff. Thus this map is a homeomorphism.

A similar argument applies to $F_t\widehat{Z}/F_{t-1}\widehat{Z}$.
\end{comment}

 \end{proof}

 \begin{cor} \label{cor: filtration of a simplicial complex}
Let $(X_i,A_i), i= 1,\cdots, m$ be finite CW pairs.   The filtrations of Definition \ref{defin: filtration of polyhedral products}  given by
\begin{enumerate}
\item $$F_tZ(K;(X,A)) = F_tZ = \cup_{wt(\sigma) \leq t} D(\sigma;(X,A)),$$ and
\item  $$F_t\widehat{Z}(K;(X,A)) = F_t\widehat{Z} = \cup_{wt(\sigma) \leq t} \widehat{D}(\sigma;(X,A))$$
\end{enumerate} have the property that  the natural quotient map
$$Z(K;X,A)) \to \widehat{Z}(K;(X,A))$$ is filtration preserving, and are natural for morphisms of simplicial
complexes $$L \to K$$ which induce an isomorphism of sets when restricted to the vertices.

Then there is a spectral sequence abutting to  $h^*(Z(K;(X,A)))$ obtained from these filtrations
for which the $E^{s,t}_0$-term is specified by

The $E^{s,t}_1$-term is specified by
\begin{enumerate}
\item $E^{s,t}_1Z(K;(X,A)) = h^{t}(F_sZ,F_{s-1}Z)=  h^{t}(\underline{X}/\underline{A})^{\sigma} \rtimes \underline{A}^{\sigma^c})$, and
\item $E^{s,t}_1\widehat{Z}(K;(X,A)) = h^{t}(F_s\widehat{Z},F_{s-1}\widehat{Z})= h^{t}( (\underline{X}/\underline{A})^{\sigma} \wedge \widehat{\underline{A}^{\sigma^c}})$.
\end{enumerate}

This spectral sequence has the following properties.

\begin{enumerate}
\item The spectral sequence is natural for embeddings of simplicial complexes $$L \subset K$$ with the same number of vertices.
\item The spectral sequence is natural for morphisms of simplicial complexes $$L \to K$$ which
are order preserving (of the left lexicographical ordering).

\

\item There is a finite filtration of $h^*(\zk)$ such that $E_{\infty}$ is the associated graded group of this filtration.

\end{enumerate}

\end{cor}

\bigskip
An immediate application of the spectral sequence is a computation of $h^*(K; \xa)$ as a ring when
 $h^*(X_i) \to h^*(A_i) $ is surjective for all $i$ and a freeness condition is satisfied.

 \begin{prop}\label{exm:splitting}

Suppose $(X_i,A_i)$ is a CW pair, $h^*(X_i) \to h^*(A_i) $ is surjective for all $i$ and $h^*(A_i) $ and $h^*(X_i/A_i)$ are free $h^*$ modules.  Then  $$\widetilde{h}^*(\zk) =
\underset{\sigma \in K }{\bigoplus} \widetilde{h}^*(  (\underline{X}/\underline{A})^{\sigma} \rtimes \underline{A}^{\sigma^c} ).$$
\end{prop}

\begin{proof}  The hypothesis implies $\widetilde{h}^*(X_i/A_i)\to h^*(X_i)$ is injective, and by the  freeness assumption, 
$h^*(C) \to h^*( Z(\Delta; \xa) ) $ in the following diagram is injective.

$$ \begin{array}{ccccccccccc}
0\leftarrow &h^*(Z(\partial \Delta; \xa)) & \leftarrow  &h^*( Z(\Delta; \xa) ) & \leftarrow & h^*(C) & \leftarrow0 \\
&\uparrow&&\uparrow&& \Vert \\
\leftarrow & h^*(Z(K_{q-1} , \xa)) & \leftarrow &h^*( Z(K_q,\xa) )&\overset{j}{ \leftarrow} & h^*(C) &\overset{\delta}{ \leftarrow}
\end{array}$$
which implies $j$ is injective. Hence $\delta=0$.  This implies the differentials in the spectral sequence are zero.

\end{proof}

\

 This is particularly interesting in the cases where the surjectivity or the freeness conditions are  satisfied  for $h^*$, but not for ordinary cohomology.  For example $h^*=K^*$ and  $(X_i,A_i) = \newline (SO(2n+1), SO(2n))$. 
 
 \

Specializing   $h^*$  to ordinary cohomology with coefficients in a field we can now prove  theorem \ref{thm:Cartan_and_split.monomorphisms}.     The surjectivity condition of Proposition \ref{exm:splitting}  implies $H^*(X_i/A_i)$ is a subring of $H^*(X_i).$  Therefore  $I$ in the following corollary is an ideal in $H^*(X_1) \otimes \cdots \otimes H^*(X_m).$

\begin{cor} Assume $H^*$ is cohomology with coefficients in a field and that  $(X_i,A_i)$ is a CW pair, such that  $H^*(X_i) \to H^*(A_i) $ is surjective for all $i$.  Then there is a ring isomorphism

$$H^*(\zk) = H^*(X_1) \otimes \cdots \otimes H^*(X_m)/I$$ where $I$ is the ideal generated by
$\widetilde{H}^*(X_{j_1}/A_{j_1}) \otimes \cdots \otimes \widetilde{H}^*(X_{j_t}/A_{j_t}), $  with $(j_1,\cdots,j_t)$ not spanning  a simplex in $K$.

\end{cor}
\begin{proof}  The hypothesis implies there is a split short exact sequence
$$0 \to \widetilde{H}^*(X_i/A_i) \to H^*(X_i) \to H^*(A_i) \to 0$$for all $*>0$.

After choosing a splitting we may write $H^*(X_i) = H^*(A_i) \oplus \widetilde{H}^*(X_i/A_i)$.

\
 The tensor product
 $H^*(X_1) \otimes \cdots \otimes H^*(X_m)$ may now be written as a sum of terms of the form
 $H^*(Y_1) \otimes \cdots \otimes H^*(Y_m)$ where $Y_i$ is $A_i$ or $X_i/A_i$.  In particular Proposition \ref{exm:splitting} implies  the natural map of
 rings $ H^*(X_1) \otimes \cdots \otimes H^*(X_m) \to H^*(\zk)$ is surjective with kernel given by the ideal $I$.
 \end{proof}
The Stanley-Reisner ring is the cohomology ring in the 
special case, $\xa = (CP^{\infty},\ast)$

\

\begin{exm}
 It follows from \cite{Coxeter} and   \cite[Example 6.40]{buchstaber.panov} that $Z(K; (D^1,S^0))$ is an orientable surface  if $K$ is the boundary of an $n$-gon. The next application of the spectral sequence is to determine the genus of this surface.  The genus is determined by the Euler characteristic of $H^*(Z(K; (D^1,S^0))$ which is the Euler characteristic of the $E^1$ page.

The boundary of an $n-$gon has $n$ $0$-simplicies and $n$  $1$-simplicies.  To compute the Euler characteristic,  the ranks of $\widetilde{H}^*(  (\underline{X}/\underline{A})^{\sigma} \rtimes \underline{A}^{\sigma^c} )$ are computed.

 \begin{enumerate}

\item
For $\sigma =\emptyset$, $H^*(  (\underline{X}/\underline{A})^{\sigma} \rtimes \underline{A}^{\sigma^c}) =  H^*(\overbrace{S^0 \times \cdots \times S^0}^n) $ where $\overbrace{S^0 \times \cdots \times S^0}^n$ is $2^n$ distinct points. \newline  So the unreduced homology has  $2^n$ $0$-dimensional classes.
\
\

\item
For $\sigma $ a 0-simplex

\begin{setlength}{\multlinegap}{100pt}
\begin{multline*}
(\underline{X}/\underline{A})^{\sigma} \rtimes \underline{A}^{\sigma^c} =\\
 S^1 \rtimes \overbrace{S^0 \times \cdots \times S^0}^{n-1}=\\
S^1 \wedge ( 2^{n-1} \mbox{ points })_+ = S^1 \wedge (\underset{2^{n-1}}{\vee} S^0) $$
\end{multline*}
\end{setlength}

So  $\widetilde{H}^*((\underline{X}/\underline{A})^{\sigma} \rtimes \underline{A}^{\sigma^c)}$
 has  $ 2^{n-1}$ $1-$ dimensional classes.  There are $n$ $0$-simplices so there are a total of   $n(2^{n-1})$ $1$-dimensional classes.

\
\
\item If $\sigma $ is a 1-simplex, the computation is similar.
 
 \begin{setlength}{\multlinegap}{100pt}
  \begin{multline*}
   (\underline{X}/\underline{A})^{\sigma} \rtimes \underline{A}^{\sigma^c} =\\
   S^2 \rtimes \overbrace{S^0 \times \cdots \times S^0}^{n-2}=\\
S^2 \wedge ( 2^{n-2} \mbox{ points })_+ = S^2 \wedge (\underset{2^{n-2}}{\vee} S^0)$$
\end{multline*}
\end{setlength}
which contributes  $2^{n-2}$ $2-$ dimensional classes.  There are $n$ $1$-simplices.  So there are a total  of $n(2^{n-2})$ 2-cells.
\end{enumerate}

So the Euler charateristic of  $E_1$ and hence of $Z(K; (D^1,S^0))$ is  $(4-n)2^{n-2}$.  This proves a theorem of Coxeter, \cite{Coxeter},  i.e. if $K$ is the boundary of an $n$-gon
 $Z(K; (D^1,S^0))$ is a surface of genus  $1 + (n -4)2^{n-3}$.

 \

  M. Davis first computed the Euler characteristic of $Z(K;(D^1,S^0))$ \cite{Davis}, but the analogous spectral sequence argument as given in the above example also gives the following result.
  
  \

 \begin{prop}\label{prop:genus} if $K$ is a simplicial complex with $m$ vertices which has $t_n$ n-simplices then

 $$\chi(Z(K;(D^1,S^0)) = \Sigma (-1)^{n+1} t_n 2^{m-n-1}.$$
 (The empty simplex is considered to be a $(-1)$-simplex,  $t_{-1}=1$).

  \end{prop}

\end{exm}
\section{\bf Computing the differentials.}\label{sec:spectral sequence}

For CW pairs  satisfying  a freeness condition, the differentials in the spectral sequence are shown to be determined by the coboundaries of the long exact sequences of the pairs $(X_i,A_i)$.    This result is used  to  compute the generalized cohomology of $\zkh$.

\

We recall the description of the $E_1$ page of the spectral sequences constructed in section \ref{sec:spectral sequence one}.

     \begin{equation} \label{thm:h*spectral sequence} \end{equation}
        
  \begin{enumerate}
\item
   $E_r(\zk) \Rightarrow \widetilde{h}^*(\zk) $ with

  $$E_1(\zk) = \underset{\sigma \in K}{\bigoplus} \widetilde{h}^*( ( \underline{X}/\underline{A})^{\sigma}
   \rtimes ( A^{\sigma^c})).$$

      \item
         $E_r(\zkh) \Rightarrow \widetilde{h}^*(\zkh) $ with

  $$E_1(\zkh) = \underset{\sigma \in K}{\bigoplus} \widetilde{h}^*( ( \underline{X}/\underline{A})^{\sigma}
   \wedge ( \widehat{A}^{\sigma^c}).$$
   \end{enumerate}

\begin{comment}
   The bi-degree $\{s,t\}$ in either spectral sequence is specified by $s=$ the index of the simplex $\sigma$ is the left-lexicographical order and $t=$ the cohomological degree.  The differential
$d_r$ defined on a class supported on a simplex of index $s$ increases $t$ by $1$ and has target a class supported on a simplex of index $s+r$.
\end{comment}

\

We next compute the differentials in the spectral sequence.  In order to do so  the strong freeness assumption, Definition \ref{def:bas}, on the pairs $(X_i,A_i)$ is imposed. 

\ 

Using the K\"{u}nneth theorem the $E_1$ page of the spectral spectral sequence converging to $\widetilde{h}^*(\zkh)$ is isomorphic  to a direct sum of $h^*-$ modules

$$ \widetilde{h}^*( ( \underline{X} /\underline{A})) ^{\sigma}
   \bigotimes   \underset{ j_i \notin \sigma}\otimes (\widetilde{h}^*(A_{j_i}))$$

 \bigskip

    \

Each summand  is a tensor product of the cohomology of $X_i/A_i$ (which is a sum of $C_i$ and $W_i$)  and the cohomology of $A_i$ (which is the sum of $E_i$ and $B_i$).   After expanding the tensor product the $E_1$ page becomes a sum of tensor products of $E_i, C_i, B_i$ and $W_i$.    Hence  any class in $ E^{s,t}_1$ is a sum of classes   $$  y_1  \otimes \cdots \otimes y_m$$
with  $y_j \in E_i, C_i, B_i$ or $W_i$.

\

There are coboundary maps $$\delta_i: H^*(A_i) \supset E_i  \to W_i \subset  H^{*+1}( X_i/A_i).$$

\begin{defin}\label{coboundary}
There is a coboundary map, $\delta$ defined on the $h^*-$generators $y_1  \otimes \cdots \otimes y_m$ of 
$$\widetilde{h}^*( ( \underline{X} /\underline{A})) ^{\sigma}
\bigotimes   \underset{ j_i \notin \sigma}\otimes (\widetilde{h}^*(A_{j_i})) $$ by the coboundary maps $\delta_i$ and the graded Leibniz rule.

\end{defin}

A monomial, $$y_1  \otimes \cdots \otimes y_m \in \widetilde{h}^*( ( \underline{X} /\underline{A})) ^{\sigma}
\bigotimes   \underset{ j_i \notin \sigma}\otimes (\widetilde{h}^*(A_{j_i}))$$ defines a simplex, $\sigma(y_1  \otimes \cdots \otimes y_m)$ as follows.  There are two indexing sets determined by $ y_1  \otimes \cdots \otimes y_m$.

$$ I_1 = \{i \in [m] |\quad  y_i \in C_i\}$$
$$I_2 =\{ i \in [m] |\quad  y_i \in W_i\}$$

\

The $h^*$-modules $C_i$ and $W_i$ are summands of $\widetilde{h}^*(X_i/A_i)$.   The $E_1$ page of the spectral sequence  is a sum of terms with factors of  $\widetilde{h}^*(X_i/A_i)$  indexed by the  simplices of $K$.   It follows that   $\sigma(y_1  \otimes \cdots \otimes y_m) = I_1 \cup I_2$   is a simplex of $K$. 

\

The  weight of $\sigma(y_1  \otimes \cdots \otimes y_m)$ in the left-lexicographical ordering of the simplices of $K$ is the filtration of $y_1  \otimes \cdots \otimes y_m$ in the spectral sequence. The simplex $\sigma(y_1  \otimes \cdots \otimes y_m)$   is called the support of $y_1  \otimes \cdots \otimes y_m$.

  \begin{thm} \label{diff}

  Asssume $(\underline{X},\underline{A})$ satisfies the strong freeness condition in definition \ref{def:bas}, then

$$y = \underset{\ell}{\Sigma} \quad  y_{1_{\ell}} \cdots y_{m_{\ell}}$$ survives to $E^{s,*}_r$ if

  \begin{enumerate}
 \item
   
   $\sigma( y_{1_{\ell}} \cdots y_{m_{\ell}})$ has weight $s$ for all $\ell$ 
   
   \

   	and,  with $\delta$ as in Definition \ref{coboundary} 
   	
   	\

   	\item 
   $$\delta(y) = \underset{t}{\Sigma} \quad  
        \overline{y}_{1_{t}} \cdots \overline{y}_{m_{t}}$$
   where 
   $$\mbox{ weight } \sigma(\overline{y}_{1_{t}} \cdots \overline{y}_{m_{t}})\geq s+r$$
   for all $t$.
   
   \end{enumerate}
    \
     Then $$d_r(y) = \underset{wt(\sigma(\overline{y}_{1_{t}} \cdots \overline{y}_{m_{t}}))=s+r}{\Sigma}  \overline{y}_{1_{t}} \cdots \overline{y}_{m_{t}}$$

   \end{thm}

\bigskip

   \begin{exm}\label{ssd2}
  We illustrate Theorem \ref{diff} for $h^*=H^*$,  $K$ a  simplicial complex on $3$ points, i.e. $K=$ $3$ distinct points,  $K=$ an edge and a disjoint point, $K=$ two edges meeting at a common vertex,  $K=$ the boundary of the $2$ simplex and finally $K=$ the $2$ simplex.  $\xa =(D^1,S^0)$.  There are the generators, $e_0 \in \widetilde{H}^0(S^0) $ and $w_1 \in \widetilde{H}^1(S^1)$.

  We first build up the spectral sequence for three distinct points  then add one edge at a time followed by the two simplex.

  \begin{itemize}
  \item Three distinct points.

 $$ \begin{array}{ccccc}
  \mbox {filtration}& 0 &1&2&3\\
   & e_0 \otimes e_0 \otimes e_0& w_1 \otimes e_0 \otimes  e_0 & e_0 \otimes w_1 \otimes e_0 & e_0 \otimes e_0\otimes w_1
 \end{array}$$

 \ref{diff} implies there is a differential from filtration $0$ to filtration $1$.

 \item Add one edge.

 $$\begin{array}{ccccccccc}

  \mbox {filtration}& 0& &1&&2&3&4\\
   & 0&& 0 && e_0 \otimes w_1 \otimes e_0 & e_0 \otimes e_0\otimes w_1&w_1 \otimes w_1 \otimes e_0
 \end{array}$$
 \ref{diff} implies a differential from filtration 2 to filtration 4.

 \item Add another edge

 $$\begin{array}{cccccccc}

  \mbox {filtration}& 0 &1&2&3&4&5\\
   & 0& 0 & 0& e_0 \otimes e_0\otimes w_1&0& w_1 \otimes e_0 \otimes w_1
 \end{array}$$
   \ref{diff}  implies a differential from filtration 3 to 5.

  \item Add another edge to form  $\partial \Delta^2$

   $$\begin{array}{cccccccc}
 \mbox {filtration}& 0 &1&2&3&4&5&6 \\
   & 0& 0 & 0& 0 &0& 0 &e_0 \otimes w_1 \otimes w_1
 \end{array}$$

 \item Add $\Delta^2$

  $$\begin{array}{ccccccccc}
 \mbox {filtration}& 0 &1&2&3&4&5&6 &7\\
   & 0& 0 & 0& 0 &0& 0 &e_0 \otimes w_1 \otimes w_1 & w_1 \otimes w_1 \otimes w_1
 \end{array}$$
 There is a differential from filtration 6 to filtration 7.
  \end{itemize}
  \end{exm}

\begin{proof}[Proof of  Theorem \ref{diff}.] The filtration in the spectral sequence is induced by the\\
 left-lexicographical order of the  simplices which are added one at a time.  The weight of a simplex is its position in this order. In particular the empty simplex has index $0$.  The proof is an induction on the index of the simplex being added.  In the proof,  $\widetilde{h}^*(F_s \widehat{K})$ is written $F_s$.  Inductively assume $F_{s-1}$ is given by Theorem \ref{diff}.  The induction starts with $F_0 = \widetilde{h}^*(A_1) \otimes \cdots \otimes \widetilde{h}^*(A_m)$.

	\

	Recall the definition of the  differentials in the spectral sequence.
	The spectral sequence is induced by the exact couple
	
	\bigskip
	\begin{center}
		
		\begin{tikzpicture}
		
		\draw [<-, thick ] (1.1,2)node[above ]{$F_t$} --( 0,-2)node[below ] {$\widetilde{h}^*(\widehat{D}(\sigma_t; (\underline{X}/\underline{A},\underline{A}))$}
		node[pos=0.4 ,right]{$\widetilde{h}^*(j)$}
		node[pos=1.05]{$\bullet$}
		node[pos=-0.01]{$\bullet$};
		\draw [<-, thick] (-2.9,2.0)node[above]{$F_{t-1}$} -- (.8,2)
		node[pos=0.5, left, above]{$\widetilde{h}^*(i)$}
		node[pos=-.025]{$\bullet$};
		\draw[->,thick] (-3,2)--(-.1,-2)node[pos=0.4, left]{$\delta$};
		
		\end{tikzpicture}
	\end{center}
	\bigskip
	Where the maps are induced by the inclusions, $i: F_{t-1} \widehat{K} \to F_t\widehat{K}$,  and the projections  $j:F_t \widehat{K} \to F_t \widehat{K}/F_{t-1}\widehat{K} = \widehat{D}(\sigma_t; (\underline{X}/\underline{A},\underline{A}))$  where $\widehat{D}(\sigma_t; (\underline{X}/\underline{A},\underline{A}))$ is defined in Lemma \ref{lemma:relhatspace}.  $\delta $ denotes the coboundary map.
	
	\vspace {5 \jot}
	
	The differential on a class $\alpha \in \widetilde{h}(\widehat{D}(\sigma_t; (\underline{X}/\underline{A},\underline{A})))$ which has survived to  $E_{s-t}$  is defined by  pulling $\widetilde{h}^*(j)(\alpha)$ back to  $\beta \in F_{s-1}$ followed by the coboundary.  i.e.  $ d_{s-t}(\alpha)= \gamma = \delta(\beta)$
	
	
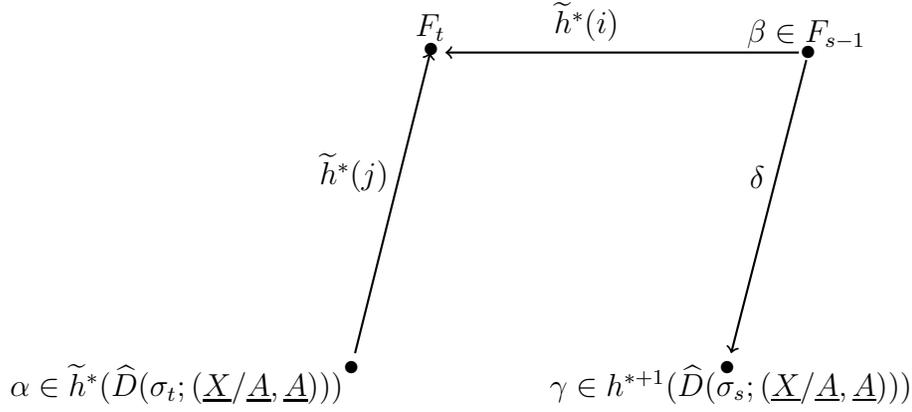
\vspace {10 \jot}
	
	\begin{tikzpicture}
	
	\draw [<-, thick ] (1,2)node[above ]{$F_t$} --( 0,-2)node[below left] {$\alpha \in \widetilde{h}^*(\widehat{D}(\sigma_t; (\underline{X}/\underline{A},\underline{A})))$}
	node[pos=0.4 ,left]{$\widetilde{h}^*(j)$}
	node[pos=1.05]{$\bullet$}
	node[pos=-0.01]{$\bullet$};
	
	\draw [<-,thick] (1.2,2)-- (5.9,2)node[pos=0.4 , left, above ]{$\widetilde{h}^*(i)$};
	
	\draw [->,thick] (6,1.9)node[above ]{$\beta \in F_{s-1}$}-- ( 5,-2)node[below ]{$\gamma \in h^{*+1}(\widehat{D}(\sigma_{s}; (\underline{X}/\underline{A},\underline{A})))$}
	node[pos=0.4, left]{$\delta$}
	node[pos=1.05]{$\bullet$}
	node[pos=-0.024]{$\bullet$};
	
	\end{tikzpicture}

	The diagram defining the differential fits into a larger diagram.

	\bigskip
	
	
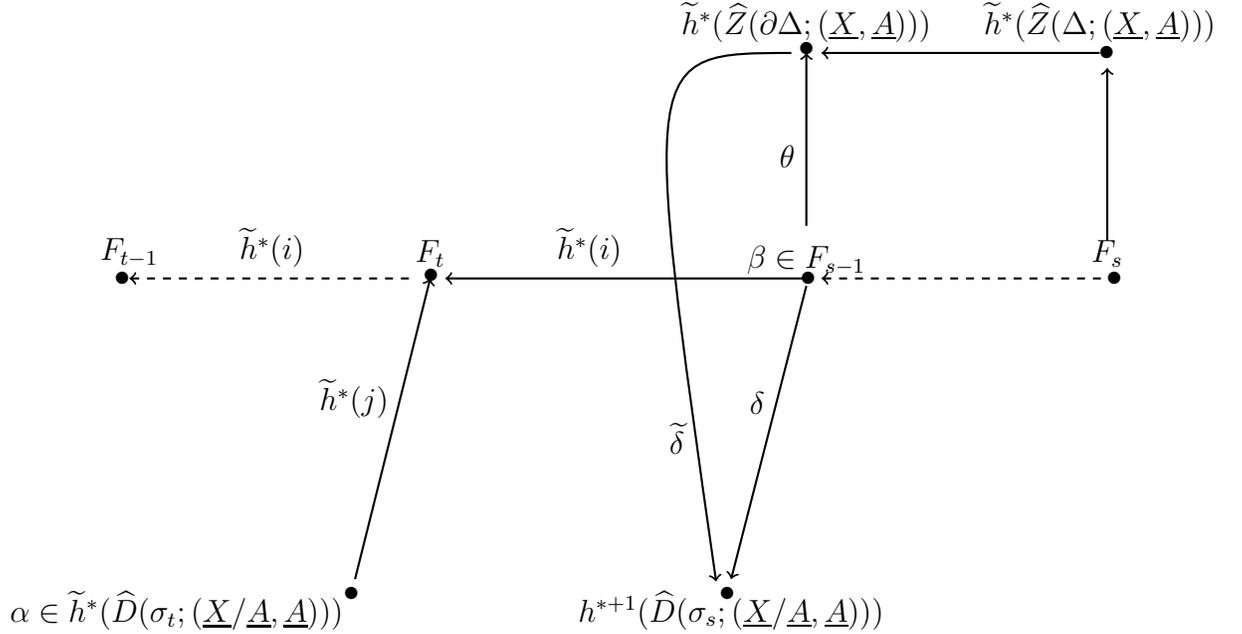
\begin{tikzpicture}
	
	\draw [<-, thick ] (1,2)node[above ]{$F_t$} --( 0,-2)node[below left] {$\alpha \in \widetilde{h}^*(\widehat{D}(\sigma_t; (\underline{X}/\underline{A},\underline{A})))$}
	node[pos=0.4 ,left]{$\widetilde{h}^*(j)$}
	node[pos=1.05]{$\bullet$}
	node[pos=-0.01]{$\bullet$};
	\draw [<-, thick, dashed] (-3,2)node[above]{$F_{t-1}$} -- (.8,2)
	node[pos=0.5, left, above]{$\widetilde{h}^*(i)$}
	node[pos=-.025]{$\bullet$};
	\draw [<-,thick] (1.2,2)-- (6,2)node[pos=0.4 , left, above ]{$\widetilde{h}^*(i)$};
	\draw [<-,thick,dashed] (6.2,2) -- (10, 2)node[above ]{$F_{s}$}
	node[pos=1.025]{$\bullet$};
	\draw [->,thick] (6,1.9)node[above ]{$\beta \in F_{s-1}$}-- ( 5,-2)node[below ]{$h^{*+1}(\widehat{D}(\sigma_{s}; (\underline{X}/\underline{A},\underline{A})))$}
	node[pos=0.4, left]{$\delta$}
	node[pos=1.05]{$\bullet$}
	node[pos=-0.024]{$\bullet$};
	\draw[->,thick] (6, 2.7) -- (6, 5)node[pos=0.4, left]{$\theta$} node[above]{$\widetilde{h}^*(\widehat{Z}(\partial \Delta; \xa))$}
	node[pos=1.024]{$\bullet$};
	
	\draw[<-,thick] (6.2, 5)-- (9.9, 5)node[above]{$\widetilde{h}^*(\widehat{Z}(\Delta; \xa))$}
	node[pos=1.025]{$\bullet$};
	\draw[<-,thick] (10,4.8)--(10,2.5);
	
	\draw[->, thick] (5.8,5) .. controls   +(left:2cm) ..  (4.8,-2.03)
	node[pos=.9, left ]{$\widetilde{\delta}$};
	\end{tikzpicture}
	
	\bigskip
	
	Where the right square is induced by the pushout diagram in Lemma \ref{lem:associated graded} and $\widetilde{\delta}$ is the coboundary map associated to the cofibration
	$$ \widehat{Z}(\partial \Delta;(\underline{X},\underline{A})) \to \widehat{Z}( \Delta;(\underline{X},\underline{A})) \to \widehat{D}(\sigma_{s}; (\underline{X}/\underline{A},\underline{A})).$$

	\

	By induction $\beta$ is a sum of terms of the form $y_1 \otimes \cdots \otimes y_m$.  We write  $(y_1 \otimes \cdots \otimes y_m)_{\ell}$ for the terms that appear in $\beta$
	
	$$\beta = \underset{\ell}{\Sigma} (y_1 \otimes \cdots \otimes y_m)_{\ell}.$$
	The differential of $\alpha$ is given by   $$d(\alpha) = \widetilde{\delta}(\theta(\underset{\ell}{\Sigma} (y_1 \otimes \cdots \otimes y_m)_{\ell}))$$
	where $\widetilde{\delta}$ is given by the formula in part (3) of Theorem \ref{diff}.
	
	\end{proof}

\

The following Lemma \ref{lem: strongly isomorphic homology and homology of smash products}, motivates the definition  of {\it strongly isomorphic $h^*-$homology}, which was also   defined in \cite{cartan}.

\begin{defin}\label{defin:maps between infinite symmetric products}

The pairs $(\underline{Y},\underline{B})$ and $(\underline{X},\underline{A})$
are said to have {\bf strongly isomorphic $h^*-$cohomology} provided

\begin{enumerate}
\item there are $h^*-$isomorphisms 
$$\beta_j: h^*(Y_j) \to h^*(X_j),$$  $$\alpha_j: h^*(B_j) \to h^*(A_j),$$ 
and 
$$\gamma_j:h^*(Y_j/B_j)\to h^*(X_j/A_j)$$

\

\
		\item there is an commutative diagram of exact sequences

		\[
		\begin{CD}
		@>{}>>  \bar{H}_i(B_j ) @>{{\lambda_j}_*}>>  \bar{H}_i(Y_j ) @>{}>> H_*(Y_j/B_j)@>{\delta}>>   \\
		&&        @VV{\alpha_j}V     @VV{\beta_j}V    @VV{\gamma_j}V \\
		@>{}>>  \bar{H}_i(A_j) @>{{\iota_j}*}>>  \bar{H}_i(X_j) @>{}>> H_*(X_j/A_j)) @>{\delta}>>  
		\end{CD}
		\]. 
		
		where  $\lambda_j: B_j \subset Y_j$, and $\iota_j: A_j \subset X_j$ are the natural inclusions.
		
		\item The maps of triples 
		$$(\alpha_j,\beta_j,\gamma_j): (H_*(Y_j), H_*(B_j), H_*(Y_j/B_j) ) \to (H_*(X_j), H_*(A_j), H_*(X_j/A_j)$$
		which satisfy conditions $1-2$ are said to {\bf induce a strong homology isomorphism}.

\end{enumerate}

\end{defin}

\
\begin{comment}
The feature of the pairs $(\underline{Y},\underline{B})$ and $(\underline{X},\underline{A})$
having strongly isomorphic $h^*-$cohomology groups  implies that the spaces $\widehat{D}(\sigma;(\underline{Y},\underline{B}))$
$\widehat{D}(\sigma;(\underline{X},\underline{A}))$ have isomorphic $h^*-$cohomology groups.
This is stated in the next lemma.

\begin{lem} \label{lem: strongly isomorphic homology and homology of smash products}
Given pointed, path-connected pairs of finite CW-complexes
 $(\underline{X},\underline{A},\underline{*})$, and  $(\underline{Y},\underline{B},\underline{*})$
 with strongly isomorphic {\bf free} $h^*-$cohomology groups, and $\sigma \in K$ any face of the simplicial complex $K$.
 Then the $h^*-$cohomology groups of $$\widehat{D}(\sigma;(\underline{X},\underline{A}))$$
 are isomorphic to the $h^*-$cohomology groups of $$\widehat{D}(\sigma;(\underline{Y},\underline{B})).$$

\

Hence there are isomorphisms  $$h^*(\zk) \to  h^*(Z(K;(\underline{Y},\underline{B})))$$

and 
$$ h^*(\zkh) \to h^*(\widehat{Z}(K;(\underline{Y},\underline{B})))$$

\end{lem}

\end{comment}

  We  will sometimes say that pairs of $h^*$ modules with maps satisfying conditions $1-3$ are strongly $h^*$ isomorphic without reference to any spaces.   The next corollary is an immediate consequence of \ref{diff}.

 \begin{cor}\label{cor:specSeqstrong isomorphism} Assume $(X_i,A_i)$ are CW pairs for all $i$ which satisfy the freeness conditions of \ref{diff}.  Then the spectral sequence  $E_r(\zk) \Rightarrow \widetilde{h}^*(\zk) $ depends only on the strong  $h^*$ cohomology isomorphism type of the pairs, $(X_i,A_i)$.  
 \end{cor}

\

Specifically the filtration, differentials and extensions depend only on $K$ and the  $h^*$ cohomology isomorphism type of the pairs, $(X_i,A_i)$.

\

A straightforward, recursive application of the splitting for the right smash,  Lemma \ref{lem: cofibrations.two},  shows that there is a homotopy equivalence

\begin{equation}\label{splittingofSS}\Sigma ( \underset{I \subset [m]}{\bigvee} [ \underset{\sigma \in K_I}{\bigvee}  ( \underline{X}/\underline{A})^{\sigma}
   \wedge (\widehat{A}^{I-|\sigma|})] )\to
\Sigma( \underset{\sigma \in K}{\bigvee}  ( \underline{X}/\underline{A})^{\sigma}
   \rtimes ( A^{\sigma^c})).
             \end{equation}

This implies  a splitting of  $E_1(\zk)$ into a  sum over $I$ of $E_1(\widehat{Z}(K_I; \xa))$.  Hence Theorem \ref{thm:decompositions.for.general.moment.angle.complexes} appears at the level of the $E_1$ pages of the spectral sequences in equation \eqref{thm:h*spectral sequence}.   The differentials respect the splitting.  The next corollary records this.

\begin{cor}\label{cor:SSsplitting} Assume $(X_i,A_i)$ are CW pairs  which satisfy the freeness conditions of \ref{diff}. Then
$$E_r(Z(K; \xa)) = \underset{I \subset [m]}{\bigoplus}E_r(\widehat{Z}(K_I; \xa_I)).$$

\end{cor}

The following example illustrates the main ideas in the proof.  The notation is as in (\ref{def:bas}).

\

\begin{equation}\label{ideaProof}\end{equation}
\begin{itemize}

\item  The splitting at   $E_1$ is a consequence of the presence of the classes $1\in B_i$ defined in \ref{def:bas} which appears in  $E_r(Z(K; \xa)) $ but not in $E_r(\widehat{Z}(K_I; \xa_I))$.

 \item    The differentials only involve  $E_i$ and $W_i$, hence the spectral sequences are modules over $B_i$ and $C_i$.
\end{itemize}

\begin{exm}   Let  $h^*=$ ordinary cohomology,  $K$ the simplicial complex of two disjoint points, $\{1,2\}$,  and $X_i = D^2 \bigvee S^3$, $A_i = S^1 \subset D^2 \subset X_i$ for $i=1,2$.  Since $i$ is determined by the coordinate in the subsequent tensor products, we omit it from the notation.

The modules  of (\ref{def:bas}) are as follows:
$$E =\Z\{e_1\}\in H^1(S^1),  \quad B = \Z\{1\} \in H^0(D^2 \bigvee S^3), $$$$ C=\Z\{c_3\} \in H^3(D^2 \bigvee S^3), \quad W=\Z\{w_2\} \in H^2((D^2/S^1) \bigvee S^3).$$

For this case, the decomposition given by equation \eqref{thm:h*spectral sequence} has the form

$$E_1(Z(K; \xa))  =
\begin{array}{cccccccccc}
H^*(A \times A) & \bigoplus & H^*(X/A \rtimes A) & \bigoplus & H^*(A \ltimes X/A)\\

\Vert && \Vert && \Vert\\
(e_1 \oplus 1)\otimes  ( 1 \oplus e_1) & \bigoplus &(c_3 \oplus w_2) \otimes (e_1 \oplus 1) & \bigoplus & (e_1 \oplus 1) \otimes (c_3 \oplus w_2)
 \end{array}
 $$
which equals

$$\begin{array}{llll}
(1 \otimes 1) \oplus
(e_1 \otimes 1) \oplus (1 \otimes e_1) \oplus(e_1 \oplus e_1) \\
\hspace{26 \jot} \bigoplus \\
(c_3 \otimes e_1) \oplus (c_3 \otimes 1) \oplus (w_2 \otimes e_1) \oplus (w_2 \otimes 1) \\
\hspace {26 \jot} \bigoplus\\
(e_1 \otimes c_3) \oplus (1 \otimes c_3) \oplus (e_1 \otimes w_2) \oplus (1 \otimes w_2)\\

\end{array}$$

Now arrange the summands according to the location of the unit

$$ \begin{array}{lllll}
(e_1 \otimes e_1) \bigoplus (c_3\otimes e_1) \oplus (w_2 \otimes e_1) \bigoplus (e_1 \otimes c_3) \oplus (e_1 \otimes w_2)\\
(e_1 \otimes 1) \bigoplus (c_3 \otimes 1) \oplus (w_2 \otimes 1)\\
(1 \otimes e_1) \bigoplus (1 \otimes c_3) \oplus (1 \otimes w_2)\\
(1 \otimes 1)
\end{array}$$

\

The first line is $E_1(\widehat{Z}(K_{\{1,2\}}; \xa))$, the second line is $E_1(\widehat{Z}(K_{\{1\}}; \xa))$ and the third line is $E_1(\widehat{Z}(K_{\{2\}}; \xa))$.  The last line is the unit in $H^0(Z(K; \xa) $ which appears as $E_1(\widehat{Z}(K_{\emptyset}; \xa))$.  This decomposition is an example of the first observation of  (\ref{ideaProof}).

\

  There is a differential in the first line from $e_1 \otimes e_1$ to $w_2 \otimes e_1$, a differential in the second line from $e_1 \otimes 1$ to $w_2 \otimes 1$ and a similar differential from $1 \otimes e_1$  to $ 1 \otimes w_2$.  The differentials are zero on $1$ and $c_3$.    This illustrates the second point of \ref{ideaProof}.

\end{exm}

  Using the units we have  splitting of the spectral sequence.

  \
  \begin{equation}\label{algebraicSplitting}
  E_r(Z(K; \xa)) = \underset{I\subset [m]}{\bigoplus} E_r(\widehat{Z}(K_I; \xa_I)) \bigotimes 1^{[m]-I}
  \end{equation}
where $1^{[m]-I}$ denotes a factor of $1$ in each coordinate in the complement of $I$.

\begin{proof} [Proof of Corollary  \ref{cor:SSsplitting}.]   First some equalities that were used in \ref{ideaProof} and are relevant to proving  \ref{cor:SSsplitting} that follow from the K\"{u}nneth formula  are listed.  We assume $X$ and $Y$  have free $h^*$ cohomology.

\begin{enumerate}

\item $$\widetilde{h}^*(X \times Y) =(\widetilde{h}^*(X)\otimes \widetilde{h}^*(Y)) \oplus( \widetilde{h}^*(X) \otimes 1)\oplus (1 \otimes \widetilde{h}^*(Y)).$$

\

This is the algebraic version of the homotopy splitting
$$ \Sigma (X \times Y) = \Sigma (X \wedge Y) \vee \Sigma (X) \vee \Sigma (Y).$$

\

\item  $$\widetilde{h}^*(X \rtimes Y) =[ \widetilde{h}^*(X)\otimes \widetilde{h}^*(Y)]  \oplus[ \widetilde{h}^*(X) \otimes 1] .$$

This is the algebraic version of Lemma \ref{lem: cofibrations.two}.

\

\item More generally,  with $\sigma = \{1,2,\cdots, n\}$.

\

$h^*( (\underline{X}/\underline{A})^{\sigma} \rtimes \underline{A}^{\sigma^c})$
$$= \widetilde{h}^*( (X_1/A_1) \wedge \cdots \wedge (X_n/A_n) \bigwedge (A_{n+1} \times \cdots \times A_m)) \bigoplus \widetilde{h}^*((X_1/A_1) \wedge \cdots \wedge (X_n/A_n)) \otimes 1 \cdots \otimes 1$$
$$=\underset{ \{i_1,\cdots i_p \} \subset \{1, \cdots, n\}}{\bigoplus} \widetilde{h}^*( (X_1/A_1)) \otimes \cdots \otimes  \widetilde{h}^*(X_n/A_n)  \otimes \widetilde{h}^*(A_{i_1}) \otimes \cdots \otimes \widetilde{h}^*(A_{i_p})) \otimes 1 \otimes \cdots \otimes 1$$
where the units appear in the factors with coordinates in $[m]$ not in the set $\{i_1,\cdots, i_p \}$.
\end{enumerate}

\

Summing over all simplices gives the isomorphism in Corollary \ref{cor:SSsplitting}.
\end{proof}
 \

\section{\bf The Cohomology of the polyhedral product. }\label{revisited}

 The freeness conditions of Section \ref{sec:spectral sequence} are assumed for $\xa$ throughout this section the goal of which  is to compute $h^*(\zkh)$ in terms of the strong $h^*$-cohomology isomorphism type of the pairs $(X_i,A_i)$ and the cohomology of sub-complexes of $K$.
  A consequence is a formula for the reduced Poincare series for $\widetilde{H}^*(\zk)$.

 \

 $K$ denotes a simplicial complex with $m$ vertices. $E_i, B_i$, $C_i$ and $W_i$ are as in Definition  \ref{def:bas}.

\

Write $E_1$ for $E_1(\widehat{Z}(K; \xa).$  

\

Recall the following  from Section
\ref{sec:spectral sequence}.

\

 \begin{enumerate}
\item
    $E_1$ is a sum of $h^*-$modules
$$ \widetilde{h}^*( ( \underline{X}/\underline{A})^{\sigma}
   \wedge ( \widehat{A}^{\sigma^c}))$$
   which is a sum of tensor products of $E_i, B_i, C_i, W_i$.
   
   \
   
 \noindent In particular a typical summand in $E_1$ may be written in the form
 
 \

   \begin{equation}
   \label{JST}
    E^J \otimes W^L \otimes C^S \otimes B^T, \end{equation} with $ J \cup L \cup S \cup T = [m]$ and   $J, L, S, T$ mutually disjoint.

 \

   \item The indexing set $L \cup S$ is a simplex in $K$ (see Definition  \ref{coboundary}). 

\

   \item  By Theorem \ref{diff} the differentials in the spectral sequence are induced by the coboundary maps, $\delta_i: \widetilde{h}^*(A_i) \to \widetilde{h}^*(X_i/A_i)$.  with $\delta_i$ mapping   $E_i$ to $W_i$.

    \end{enumerate}

   \

We next show that $E_r$ is a sum of simpler spectral sequences.  To this end fix $S$ and $T$.  The next definition, which is simply a  reparametrization of $C^S \otimes B^T$ is for convenient bookkeeping.
\begin{defin}
 Let $I \subset [m], \quad I=(i_1,\cdots , i_p)$.

     Let $\sigma \subset I$ be a simplex in $K$.  Define

  $Y^{I,\sigma} =  Y_1 \otimes \cdots \otimes Y_p \mbox{where} $

 $ Y_t=  \begin{cases}

  C_{i_t} & \mbox{ if }  i_t \in \sigma \\

  B_{i_t} & \mbox{ if } i_t \notin \sigma
  \end{cases}$

\end{defin}

 In the notation of (\ref{JST})  $\sigma=S$ and  $I = S \cup T$.
 
 \

 By definition, $Y^{I,\sigma}$ is fixed because $S$ and $T$ are fixed .  Now  consider the sum
  $$\Big[ \underset{\substack{J\cup L \cup I=[m] \\ L \cup \sigma \in K}}{\bigoplus}E^J\otimes W^L \Big]\otimes Y^{I,\sigma}$$
    which is a sub-sum of the $E_1$ page of the spectral sequence.
    Since $J, L$ and $I$ are mutually disjoint, and $L \cup \sigma \in K$ it follows that  $L$ is a simplex of $K$ belonging to the link of $\sigma$ in the complement of $I$, which is now defined.

 \begin{defin}\label{def:N(Isigma)} If $I \subset [m]$ and  $\sigma \in K$, $\sigma \subset I$ then the link of $\sigma$ in the complement of $I$,
 $N(I,\sigma),$ is the simplicial complex, $N(I,\sigma)$ on vertex set $[m]-I$ such that

  $$\overline{\sigma} \in N(I,\sigma)  \Leftrightarrow \overline{\sigma} \cup \sigma \in K.$$
  \end{defin}
  Note that
  \begin{enumerate}

  \item $N(I,\sigma)$ is indeed a simplicial complex since $\sigma^{\prime} \subset \overline{\sigma} $ implies $\sigma^{\prime} \cup \sigma \in K$ which implies $\sigma^{\prime} \in N(I,\sigma)$.
  
  \
  
\item   If $N(I,\sigma) =\emptyset$   then $\sigma \cup \{v\} \notin K$ for all $v \in [m]-I$.  This implies \newline $N(I,\sigma) =\emptyset \Leftrightarrow \sigma$ is a maximal simplex in $K$.

\

\item $N(I, \emptyset)= K_{[m]-I}$.

\

\item $N(|\sigma|,\sigma) = lk_{\sigma}(K)$, the link of $\sigma$ in $K$.

 \end{enumerate}

\

Since all differentials take place between the terms $E^J$ and $W^L$ in $E_r$ it follows that  for  fixed $I,\sigma$ 
$$\Big[ \underset{\substack{J\cup L \cup I=[m] \\ L \cup \sigma \in K}}{\bigoplus}E^J\otimes W^L \Big]\otimes Y^{I,\sigma}$$
is a sub spectral sequence of $E_r$ with all differentials taking place within the brackets. In particular 

\

$$ \Big[ \underset{\substack{J\cup L \cup I=[m] \\ L \cup \sigma \in K}}{\bigoplus}E^J\otimes W^L \Big]$$is the $E_1-$page of a spectral sequence.  

\

 This spectral sequence is next identified as a spectral sequence
$$\Big[ E_r(\widehat{Z}(N(I,\sigma); (\underline{CV},\underline{V})))\Big]$$ for some collection of CW complexes, $\{V_i\}$.

\

Recall that $h^*(A_i) = E_i \oplus B_i$ for free $h^*-$modules $E_i, B_i$.  Let 
$$\{e^{\ell}_i\}$$ be a set of generators for $E_i$.

\

\begin{defin}
Set $$V_i = \bigvee S^{|e^{\ell}_i|}$$
\end{defin}
\noindent where $|e^{\ell}_i|$ is the dimensions of  generator $e^{\ell}_i$.

\

The $V_i$ are constructed so that $(0, \widetilde{h}^*(V_i))$ is strongly $h^*$ cohomology isomorphic to $(0,E_i)$.  {\it Strongly $h^*$ cohomology isomorphic} is defined in definition \ref{defin:maps between infinite symmetric products} and the paragraph before corollary \ref{cor:specSeqstrong isomorphism}.  By construction the spectral sequence $$ E_r(\widehat{Z}((N(I,\sigma); (\underline{CV},\underline{V}))$$is isomorphic to the spectral sequence 
$$\Big[ \underset{\substack{J\cup L \cup I=[m] \\ L \cup \sigma \in K}}{\bigoplus}E^J\otimes W^L\Big].$$  The spectral sequence converges to $h^*(\widehat{Z}(N(I,\sigma));(\underline{CV},\underline{V}))$. 
 By the  wedge lemma,  \ref{thm:null.A} 
$$ h^*(\widehat{Z}(N(I,\sigma); (\underline{CV},\underline{V}) = h^*(\Sigma |N(I,\sigma|) \otimes E^{[m]-I}.$$

\

\

Combining these results gives the following calculation of the cohomology of the polyhedral smash product functor, which by the algebraic splitting theorem, \ref{cor:SSsplitting}, gives the cohomology of the polyhedral product functor for CW pairs satisfying the freeness condition.  The difference between the spectral sequence converging to $h^*(Z(K; \xa)$ and the one converging to $\widetilde{h}^*(\widehat{Z}(K; \xa)$ is describe in  the  proof of  Corollary \ref{cor:SSsplitting}.

  \

The extensions in the spectral sequence converging to 
$$h^*(\zkh)$$ appear as extension is the spectral sequence converging to $h^*(\Sigma |N(I,\sigma|) \otimes E^{[m]-I}$.  
So we may write   $ E^{[m]-I}\otimes \widetilde{h}^*(\Sigma |N(I,\sigma|)$ below rather than the associated graded group.

\noindent \begin{thm}\label{Ps}

\begin{enumerate}

\

\item
            $$\widetilde{h}^*(\zkh)) = \underset{I\subset [m], \sigma \in K}{\bigoplus} E^{[m]-I} \otimes \widetilde{h}^*(\Sigma|N(I,\sigma)|) \otimes Y^{I,\sigma}.$$as $h^*-$modules
    with $  \widetilde{h}^*(\Sigma \emptyset)  =1$.   The factors, $B_i$ of $Y^{I,\sigma}$ are   subsets of the \textbf{reduced  cohomology} of $A_i$.  i.e. $1\in B_i$ as defined in Definition \ref{def:bas} does not appear.

    \

\item  Similarly $$ h^*(\zk)) = \underset{I \subset [m], \sigma \in K}{\bigoplus} E^{[m]-I} \otimes \widetilde{h}^*(\Sigma|N(I,\sigma)|) \otimes Y^{I,\sigma}. $$as $h^*-$modules
 with $  \widetilde{h}^*(\Sigma \emptyset)  =1$, but now  the factors, $B_i$ of $Y^{I,\sigma}$ are  subsets of the \textbf{un-reduced cohomology} of $A_i$.  i.e. $ 1$ as defined in Definition \ref{def:bas} may appear in a coordinate of $Y^{I,\sigma}$.

\end{enumerate}
     \end{thm}

A convenient reformulation of Theorem \ref{Ps} is given in the next corollary.

\begin{cor}\label{srstar} 
	$$h^*=\underset{}{\underset{  \underset{I \cup J =[m] } {I \cap J =\emptyset}  }{\bigoplus} }\widetilde{h}^*(\Sigma | N_J|) \otimes E^J \otimes h^*(X^I)/(R)$$ where $I=\{i_1, \cdots, i_t\}$, $X^I=X_{i_1} \times \cdots \times X_{i_t}$,
	and   $(R) $ is the ideal  generated by $ C^{S} \subset H^*(X^I)$ where  $S$ is not simplex of $K$.
\end{cor}

  The reduced Poincare series   for $H^*(\zkh) $ is recorded as the following corollary.

  \begin{cor}\label{spectralsequencePS}

            $$\overline{P}(H^*(\zkh)) = \underset{I, \sigma}{\Sigma} \quad  t\overline{P}(Y^{I,\sigma})\times   \overline{P}(H^*(|N(I,\sigma)|) \times \overline{P}(E^{[m]-I}).$$

   where $\overline{P}(H^*(\emptyset))= 1/t$.

  \end{cor}
  
 \begin{remark}  The Poincare series was  first computed in \cite{cartan}.

 \end{remark}

   \begin{exm} \label{SSexample} We illustrate Corollary \ref{spectralsequencePS} with the example of  $K $ a simplicial complex with 3 vertices and  edges $\{1,3\}, \{1,2\}$.    $H^*(X) =\Z\{ b_4, c_6\}, \quad H^*(A)=\Z\{e_2, b_4\} $.   The cases in the example are indexed by the $I$ in Corollary \ref{spectralsequencePS} starting with the empty set and building up to $I=\{1,2,3\}$.  For each $I$ there are the sub cases indexed by the simplices $\sigma \subset I$.
   \begin{itemize}
   \item For $I=\emptyset$.  the only possible simplex,  $\sigma$, is the empty set and $N(I,\emptyset$) is contractible.  So there is no contribution to the Poincare series.

   \item The next case is $I=\{1\}$. There are two possible simplices,  namely $\sigma=\emptyset$ and $\sigma=\{1\}$.
   \begin{enumerate}
     \item $\sigma = \emptyset$.

      In this case $Y^{I,\emptyset}= b_4.$  and $|N(I,\emptyset)| = |\{ \{2\},\{3\}\}|=S^0$ which contributes 

      $$t\overline{P}(Y^{I,\sigma}) \overline{P}(H^*(|N(I,\sigma)|) \overline{P}(E^{[m]-I}) =t (t^4)(t^2)^2=t^9$$to the Poincare series.
      
      \item $\sigma=\{1\}$.

       In this case $Y^{I,\sigma} = c_6.$ and     $|N(I,\sigma)|= |\{2\},\{3\}| =S^0$.  Thus the term 

       $$t\overline{P}(Y^{I,\sigma}) \overline{P}(H^*(|N(I,\sigma)|) \overline{P}(E^{[m]-I})=t(t^6)(t^2)^2 = t^{11}$$ is contributed to the Poincare series.

  \end{enumerate}

  \item Similarly for $I=\{2\}$ there are the cases 

  \begin{enumerate}
  \item $\sigma=\emptyset$ with  $N(I,\sigma)=$ the edge $ \{1,3\}$ which is contractible.
  \item $\sigma = \{2\}$ with  $N(I,\sigma) = \{1\}$ which is also contractible.

  \end{enumerate}
  \item $I= \{3\}$.  This is similar to $I=\{2\}$.

  \begin{enumerate}
  \item $\sigma=\emptyset$,  $N(I,\sigma)=$ the edge $ \{1,2\}$ which is contractible.
  \item $\sigma = \{3\}$,  $N(I,\sigma) = \{1\}$ which is contractible.

  \end{enumerate}

  \item For $I= \{1,2\}$ there are $4$ possible simplices

  \begin{enumerate}
  \item $\sigma= \emptyset$ with  $N(I,\emptyset)= \{3\}$ which is contractible.

  \item $\sigma = \{1\}$ with   $N(I,\sigma) = \{3\}$ which is contractible.

  \item $\sigma = \{2\}$, with $Y^{I,\sigma} = b_4 \otimes c_6$ and $N(I,\sigma)=\emptyset$ which contributes 
  $$  t\overline{P}(Y^{I,\sigma}) \overline{P}(H^*(|N(I,\sigma)|) \overline{P}(E^{[m]-I})=(t^6t^4)(t^2)=t^{12}.  $$to the Poincare series.
  
\item  $\sigma=\{1,2\}$

  $Y^{I,\sigma}= c_6\otimes c_6$, \quad   $N(I, \{1,2\}) = \emptyset$.  So this case  contributes $t^{14}$ to the Poincare series.

  \end{enumerate}
  \item $ I= \{1,3\}$ is identical to $I=\{1,2\}$ so we get a contribution of $t^{12}$ and $t^{14}$ to the Poincare series.

  \item $I=\{2,3\}$

  \begin{enumerate}
  \item  The cases  there are $3$ possible simplices: $\sigma = \emptyset, \quad \{2\}$, and $\sigma= \{3\}$.  For all $3$ simplices   $N(I,\sigma)=\{1\}$ which is contractible.

  \end{enumerate}

  \item Finally $I=\{1,2,3\} =K$.   For all $\sigma$, $N(I,\sigma)=\emptyset$.

  The sub simplices of $K$ contribute to the Poincare series as follows:

  \begin{enumerate}
  \item

  \begin{enumerate}
  \item
  $\sigma = \{1\}, Y^{I,\sigma}= c_6 \otimes b_4 \otimes b_4$
  \item
  $\sigma =  \{2\}, Y^{I,\sigma}= b_4 \otimes c_6 \otimes b_4$

  \item   $\sigma = \{3\}, Y^{I,\sigma}= b_4 \otimes b_4 \otimes c_6$
 \end{enumerate}
  each contributes $  t^{14}$ to $\overline{P} .$

   \item
   \begin{enumerate}
  \item $\sigma = \{1,2\},   Y^{I,\sigma}= c_6 \otimes c_6 \otimes b_4$
  \item   $\sigma =  \{1,3\} , Y^{I,\sigma}=  c_6 \otimes b_4 \otimes c_6$
  \end{enumerate}
   each contributes $t^{16} $to  $\overline{P}.$

  \item $\sigma = \emptyset,Y^{I,\sigma}=b_4 \otimes b_4 \otimes b_4$, contributes $t^{12} $ to $\overline{P}$

  \end{enumerate}

  Adding all the terms we get have the Poincare series for the cohomology of \\
  $\zkh$.

  $$\overline{P}(H^*(\zkh)) = t^9+t^{11}+3t^{12}+5t^{14}+2t^{16}.$$

   \end{itemize}

   \end{exm}

   \

   The next two corollaries describe  summands in $h^*(\zkh)$.  The summand in Corollary \ref{cor:summand} depends on the  $E_i$.  The summand in corollary \ref{cor:5.8} is natural since $C_i$ is the kernel of the map $H^*(X_i) \to H^*(A_i)$ which is a functor of the pairs $(X_i,A_i)$.

   \begin{cor}\label{cor:summand} $\widetilde{h}^{*}(\Sigma|K|) \otimes E_1 \otimes  \cdots \otimes E_m$ is a summand $\widetilde{h}^*(\zkh)$.
   \end{cor}

 \begin{proof}   This corresponds  to the summands with $I=\emptyset$ in Theorem \ref{Ps}.  Specifically if $I=\emptyset$ then $\sigma=\emptyset$ and $N(\emptyset, \emptyset)=K$.
    \end{proof}

\begin{remark} Corollary \ref{cor:summand} generalizes \cite[Theorem 1.12]{bbcg3} which describes the cohomology of $\widehat{Z}(K; (\underline{CX},\underline{X}))$.   In this case $B_i=C_i=0$ and  $E_i=\widetilde{H}^*(X_i).$   The only summand  is  $I=\emptyset$. \end{remark}

\begin{comment}

Corollary  \ref{srstar} appears to prove  Theorem 1.12 of \cite{bbcg3}  which describes  the cohomology of $H^*(Z(K; (\underline{CX}, \underline{X})))$ as the  subring of

$$H^*(Z(K;(D^1,S^0))\otimes H^*(X_1) \otimes \cdots \otimes H^*(X_m)$$
defined by

\begin{equation}\label{d1s0}
\underset{I \subset [m]}{\bigoplus}H^*(\Sigma K_I|) \bigotimes \widetilde{H}^*(\widehat{X}^I) \otimes 1^{[m]-I} \subset H^*(Z(K;(D^1,S^0))\otimes H^*(X_1) \otimes \cdots \otimes H^*(X_m)
 \end{equation}
 
 \
 
where $\widehat{X}^I = X_{i_1} \wedge \cdots \wedge X_{i_t}$ and $1^{[m]-I}$ denotes a factor of $1$ in the remaining coordinates.
  
\end{comment}

  However Theorem   1.12 of  \cite{bbcg3}   is not a consequence of Corollary \ref{cor:summand} since        the complicated bookkeeping involved to evaluate all the differentials in the spectral sequence were subsumed by        \cite[Theorem 1.12]{bbcg3}  which  was used to prove Theorems \ref{Ps} and therefore Corollary \ref{cor:summand}.  
 \

   \begin{cor}\label{cor:5.8}

Let $I \subset h^*(X_1 \times \cdots \times X_m) $ be the ideal generated by $C_{i_1}\otimes \cdots \otimes  C_{i_q} $ where $(i_1,\cdots, i_q)$ is not a simplex in $K$.  Then
$h^*(X_1 \times \cdots \times X_m)/(I) $ is a sub-ring of   $\widetilde{h}^*(\zk)$.
  \end{cor}

\begin{proof}  There are the maximal summands, $Y^{[m],\sigma}$.
$$ \underset{\sigma \in K}{\bigoplus} Y^{[m],\sigma} \simeq  h^*(X_1 \times \cdots \times X_m)/(I) $$
as $h^*$ modules.
   The inclusion,  $$ \iota: \zk \to X_1 \times  \cdots \times X_m $$ induces a  surjection in cohomology onto $\underset{\sigma \in K}{\bigoplus} Y^{[m],\sigma} $.     To prove corollary \ref{cor:5.8} it suffices to show that $I $  is isomorphic to the kernel of $h^*(\iota)$.

\

We may write $h^*(X_i) = B_i \oplus C_i$.  The tensor product
 $h^*(X_1) \otimes \cdots \otimes h^*(X_m)$ may now be written as a sum of terms of the form
 $S_1 \otimes \cdots \otimes S_m$ where $S_i$ is $B_i$ or $C_i$.  The map of
 rings, $ h^*(X_1) \otimes \cdots \otimes h^*(X_m) \to h^*(\zk)$ is surjective with kernel given by the ideal $I$.

   \end{proof}
   
   \
   
\begin{remark}\label{rem:naturality}
   The spectral sequence is natural with respect to maps of pairs
$$(X_i,A_i) \to (Y_i,D_i)$$ 
 (In fact it is a functor of strong $h^*-$cohomology maps of pairs as described in Corollary  \ref{cor:specSeqstrong isomorphism}).
However the description of the cohomology of $\zk$  given in  Theorem \ref{Ps} is not natural.  It depends on the choice of splitting of $h^*(X_i)$ as $B_i \oplus C_i$ and splitting of $h^*(A_i)$ as $E_i \oplus B_i$.   We now describe how the decomposition in Theorem \ref{Ps} and the naturality of the spectral sequence interact.

\

To this end suppose there are  maps of long exact sequences

\begin{equation}\label{eq:naturalityofdecomposition}
\begin{array}{ccccccccccc}
\cdots \overset{\delta}{\leftarrow}&h^*(A_i) &\leftarrow&h^*(X_i)&\leftarrow & \widetilde{h}^*(X_i/A_i) & \leftarrow \cdots\\
&\uparrow{g_i} && \uparrow{f_i} && \uparrow\\
\cdots \overset{\delta}{\leftarrow} &h^*(D_i)&\leftarrow&h^*(Y_i) &\leftarrow & \widetilde{h}^*(Y_i/D_i) &\leftarrow \cdots
\end{array}
\end{equation}

\

Diagram \ref{eq:naturalityofdecomposition} induces a map of spectral sequences

$$E_r(Z(K;(\underline{Y},\underline{D}))) \overset{(f,g)^*}{\to} E_r(\zk) $$and a map

$$h^*(Z(K;(\underline{Y},\underline{D}))) \overset{\ell}{\to} h^*(\zk).$$
\

The map $\ell$ is now described in terms of the decompositions given by  Definition \ref{def:bas}.  

\

For each $i$ there is a decomposition of the $h^*-$modules in the the top row

\

$$h^*(A_i) = E_i \oplus B_i, \quad h^*(X_i) = B_i \oplus C_i.$$

\

The bottom row has a  corresponding decomposition  

\

$$h^*(D_i) =E_i^{\prime} \oplus B_i^{\prime}, \quad h^*(Y_i) =   B_i^{\prime} \oplus C_i^{\prime}.$$

\

Suppose   $\alpha \in h^*(Z(K;  (\underline{Y}, \underline{D})))$ is a class appearing in a summand  $$ (E^{\prime})^{[m]-I} \otimes \widetilde{h}^*(\Sigma|N(I,\sigma)|) \otimes (Y^{\prime})^{I,\sigma}$$ 
of the decomposition of  $h^*(Z(K;(\underline{Y},\underline{D}))) $ given by Theorem \ref{Ps}.

\

Specifically  
 \begin{equation}
 \label{alphadecomposition}
 \alpha=\underset{J}{\bigotimes} (e_j^{\prime}) \otimes n \otimes \underset{|\sigma|}{\bigotimes} ( c_s^{\prime})\otimes  \underset{L}{\bigotimes} (b^{\prime}_{\ell}) 
 \end{equation}
  where $I=L \cup |\sigma|$, $J=[m]-I$ and $n \in \widetilde{h}^*(\Sigma|N(I,\sigma)|)$

 \
 
 In order to compute  $\ell(\alpha) \in h^*(\zk)$ we note that the decompositions of $h^*(A_i)$ and $h^*(X_i) $ into $E_i, B_i$ and $C_i$ imply unique representations $$ g_i(e_i^{\prime})= e_i+\overline{b}_i, \hspace {.2 in}   f_i(b_i^{\prime})=b_i+\overline{c}_i \mbox{   and }   f_i(c_i^{\prime}) = c_i$$
where  $e_i \in E_i,\quad b_i$ and $ \overline{b}_i \in B_i, \quad$ $\overline{c}_i$ and $c_i \in C_i$.

\

\

Formally substitute  $e_i+\overline{b}_i$ for $e^{\prime}_i$, $b_i+\overline{c}_i$ for $b_i^{\prime}$ and $c_i$ for $c^{\prime}_i$ in (\ref{alphadecomposition}).  The resulting expression is a sum of terms with factors  $e_i,$ $ b_i,$ $ \overline{b}_i,$ $\overline{c}_i,$ $c_i$ and $n \in \widetilde{h}^*(\Sigma|N(I,\sigma)|)$.  Each summand determines a summand in $\ell(\alpha)$.  There are a number of cases.

\

The easiest case is the summand without any over-lined factors.  For this term  the map of spectral sequences  $(f,g)^*$  respects the decomposition of Definition \ref{def:bas} at the $E_1$ page and contributes the summand

$$\underset{J}{\bigotimes} (e_j) \otimes n \otimes \underset{|\sigma|}{\bigotimes} ( c_s) \underset{L}{\bigotimes} (b_{\ell})$$ to $\ell(\alpha)$.

\

Now suppose there are terms of the formal sum with non-zero $\overline{b}_i$ factors    i.e. there is an indexing set $Q \subset J$ where $\overline{b}_q$ is not zero for $q \in Q$.  In this situation  there are formal summands  in

\

$$E^{J\setminus Q}  \otimes n \otimes C^{|\sigma| } \otimes B^{L \cup Q}.$$

\

In terms of the decomposition of Theorem \ref{Ps} these terms contribute  classes in 

$$E^{J \setminus Q} \otimes \widetilde{h}^*(\Sigma |N(I \cup Q,\sigma|)) \otimes C^{|\sigma|} \otimes B^{ L \cup Q}$$ to $\ell(\alpha)$ (recall that $I = L \cup |\sigma|$).

\

The simplicial complex $N(I \cup Q,\sigma)$ is a sub-simplicial complex of $N(I ,\sigma)$.  To prove this suppose   $\tau \in N(I \cup Q,\sigma)$ then $\tau \cup \sigma \in K$ and $|\tau| \subset J \setminus Q \subset J$ so $\tau \in N(I,\sigma)$.

\

The formal summand
$$E^{J\setminus Q}  \otimes n \otimes C^{|\sigma| } \otimes B^{L \cup Q}.$$

\
contributes

$$E^{J\setminus Q}  \otimes \iota^*(n)  \otimes C^{|\sigma| } \otimes B^{L \cup Q}.$$ 

\

to $\ell(\alpha)$   where $\iota^*$ is induced by the inclusion

\

$$\iota^*: \widetilde{h}^*(\Sigma |N(I,\sigma)|) \to \widetilde{h}^*(\Sigma |N(I \cup Q,\sigma)|).$$

\

Indeed $\iota^*$ at the cochain level is the map which sends the dual of a simplex, $\tau$ to zero if $\tau$ is not a simplex in $N(I \cup Q,\sigma)$ and to the dual of $\tau$ if $\tau$ is a simplex in 
$N(I \cup Q,\sigma)$.   This agrees with the map $(f,g)^*$.

\

 Finally suppose there are terms of the formal sum with non-zero $\overline{c}_i$ factors.  i.e there are indexing sets $P \subset L$ where $\overline{c}_{p}\neq 0$ for $p \in P$.  These formal summands are classes in

\

$$E^J \otimes n \otimes C^{|\sigma| \cup P} \otimes B^{ L\setminus P}.$$

\

In terms of the decomposition of Theorem \ref{Ps} these terms will  contribute  classes in 

$$E^J \otimes \widetilde{h}^*(\Sigma |N(I,\sigma \cup P|)) \otimes C^{|\sigma| \cup P} \otimes B^{ L\setminus P}.$$

 \
 
    We shown that    $N(I,\sigma \cup P)$ is a sub complex of $N(I,\sigma)$ if  $\sigma\cup P$ is a simplex in $K$ (otherwise  thie summand lies in the zero group).  Suppose  $\sigma\cup P$ is a simplex in $K$, say $\tau$.   Let $ \rho \in N(I,\tau)$.  Then $\rho$ has vertices in $J$ and $\rho \cup \sigma \cup P$ is a simplex in $K$ which implies $\rho \cup \sigma$ is also a simplex in $K$ and $\rho \in N(I,\sigma)$.   The summands of $\ell(\alpha)$ with $\overline{c}_p$ factors are represented by classes which jump filtration in the spectral sequence with $n$  replaced by $\iota^*(n)$  ($\iota: N(I,\tau) \to N(I,\sigma)$  the inclusion).

\
\begin{comment}
The claim follows since for $p \in P$  the factors $b_p^{\prime}$ in  $\alpha$ are supported on $h^*(D_p)$  in $E_1$.  The map of spectral sequences induced by (\ref{eq:naturalityofdecomposition})  sends 
$h^*(D_p)$ to $h^*(A_p)$.  In particular the terms containing the  classes $c_p$ are zero in the spectral sequence for $h^*(\zk)$.   This means that these terms have jumped filtration and are represented by terms whose support contains factors $h^*(X_p/A_p)$ corresponding to the extra vertices in $\tau$.  
\end{comment}
\

 This description of $\ell$ will be applied in  (\ref{productmap}).
\end{remark}

\

 \section{\bf Products}\label{sec: products}

The purpose of this section is to describe the ring structure in $H^*(\zk ,\mathbb{ R})$ with $\mathbb{R}$ a commutative ring.   (there are similar results for $h^*(\zk)$). As in the previous sections $(X_i,A_i)$ are assumed to be based CW pairs which satisfy the freeness condition of Definition \ref{def:bas}.  We write $H^*(\zk)$ for  $H^*(\zk ,\mathbb{ R})$ in the sequel.

\

Recall the computation of $H^*(\zk)$ as an $h^*-$module  given by \ref{Ps}.

$$ h^*(\zk)) = \underset{I \subset [m], \sigma \in K}{\bigoplus} E^{[m]-I} \otimes \widetilde{h}^*(\Sigma|N(I,\sigma)|) \otimes Y^{I,\sigma}. $$

\

In particular $h^*(\zk)$ is generated, as an $h^*-$module by monomials
$$n \otimes a_1 \otimes \cdots \otimes a_m $$
with $a_i \in E_i, B_i$ or $C_i$ and 

$$n \in \begin{cases}
 \widetilde{H}^*(\Sigma |N(I,\sigma)|) & \mbox{ if }  
          \sigma = \{i| a_i \in C_i\}\subset [m] \mbox{ is a simplex in } K \\

0 & \mbox{ otherwise. }
\end{cases}$$where  
  $I=\{i| a_i \in C_i \mbox{ or } B_i\} \subset [m] $ and $N(I,\sigma)$ is defined in Definition \ref{def:N(Isigma)}.

\

To describe $$H^*(\zk)$$ as a ring it suffices to define a paring on the summands of Theorem \ref{Ps}

\

$$ \Big[ E^{[m]-I_1} \otimes \widetilde{H}^*(\Sigma|N(I_1,\sigma_1)|) \otimes Y^{I_1,\sigma_1}\Big] \otimes  \Big[E^{[m]-I_2} \otimes \widetilde{H}^*(\Sigma|N(I_2,\sigma_2)|) \otimes Y^{I_2,\sigma_2}\Big] \overset{\cup}{\to} H^*(\zk)$$

\

Specifically suppose that
$$\alpha = n_{\alpha} \otimes a_1 \otimes \cdots \otimes a_m \in  E^{[m]-I_1} \otimes \widetilde{H}^*(\Sigma|N(I_1,\sigma_1)|) \otimes Y^{I_1,\sigma_1}$$

\
where $n_{\alpha}\in \widetilde{H}^*(\Sigma|N(I_1,\sigma_1)|)$ and $a_i \in E_i, B_i$ or $C_i$.

\

$$\gamma = n_{\gamma} \otimes g_1 \otimes \cdots \otimes g_m \in E^{[m]-I_2} \otimes \widetilde{H}^*(\Sigma|N(I_2,\sigma_2)|) \otimes Y^{I_2,\sigma_2}$$

\

where $n_{\gamma} \in \widetilde{H}^*(\Sigma|N(I_2,\sigma_2)|) $ and $g_i \in E_i, B_i $ or $C_i$.

\

We will describe $\alpha \cup \gamma$ in terms of a coordinate wise multiplication of $a_i$ and $g_i$ and a paring 

\begin{equation}
\label{pairing}H^*(\Sigma | N(I_1, \sigma_1)|) \otimes H^*(\Sigma |N(I_2,\sigma_2)|) \to H^*(\Sigma |N(I_3, \sigma_3)|)\end{equation}
 where $(I_3,\sigma_3)$ will be defined in terms of $(I_1,\sigma_1)$ and $(I_2, \sigma_2)$
 
 \
 
 The pairing, (\ref{pairing}) will be defined in terms of the $\ast-$product defined in \cite{bbcg3} which we now recall.

\

Writing $ Z(K)$ for $\zk$ and $\widehat{Z}(K_I)$ for $\widehat{Z}(K_I; (\underline{X}, \underline{A})_I)$,   {\it partial diagonals}

$$\Delta^{J,L}_I:\widehat{Z}(K_I)\to
\widehat{Z}(K_J)\wedge\widehat{Z}(K_L)$$  ($J\cup L=I$)
are defined  which fit into  a diagram

\[\begin{CD}
Z(K) &@>\widehat{\Delta}>>
&Z(K)\wedge Z(K) \\
@V \Pi VV&&@VV\Pi \wedge \Pi V\\
\widehat{Z}(K_I) &@> \Delta^{J,L}_I >>&
\widehat{Z}(K_J)\wedge\widehat{Z}(K_L).
\end{CD} \]
where $\widehat{\Delta}$ is the reduced diagonal and $\Pi$ is the projection.

\

The definition of the  partial diagonals and projections are as follows  (the notation is as in Definitions \ref{defin:gmac} and \ref{defin:smash.product.moment.angle.complex}).

\begin{enumerate}
	\item For $I \subset [m]$ and $\sigma \in K$ there is the projection followed by the collapsing map
	$$\pi: D(\sigma) \to D(\sigma \cap I) \to \widehat{D}(\sigma \cap I).$$  These composites  are compatible with the maps in the colimit and induce the vertical maps
	$$\Pi : Z(K) \to \widehat{Z}(K_I).$$
	
	\
	
	\item Let $W^{J,L}_I$ be defined by
	$$
	W^{J,L}_I =
	\begin{cases}
	Y_i & \mbox{ if } i \in I - J \cap L \\
	Y_i \wedge Y_i & \mbox{ if } i \in J \cap L
	\end{cases}
	$$
	where $Y_i$ is either $X_i$ or $A_i$ as in Definition \ref{defin:smash.product.moment.angle.complex}.
	
	\
	There is a homoemorphism
	
	$$Sh: W^{J,L}_I \to \widehat{Y}^J \wedge \widehat{Y}^L$$  given by the evident shuffle which is compatible with the maps in the colimit.
	
	\item Define
	$$ Y^I \to \widehat{Y}^J \wedge \widehat{Y}^L$$ by  first mapping into $W^{J,L}_I$ by
	
	\
	
	\begin{enumerate}
		\item the identity of $Y_i$ if $i \in I - J \cap L$  or
		
		\item the diagonal of $Y_i$ if $ i \in J \cap L$
	\end{enumerate}
	
	\
	
	\noindent followed by $Sh$.  The maps induce a map of colimits which define the partial diagonal, $\Delta^{J,L}_I$.
	
\end{enumerate}

Given cohomology classes $u\in H^p(\widehat{Z}(K_J)), v\in H^q(\widehat{Z}(K_L))$, the {\it $\ast-$product } is defined by
\begin{equation} \label{starproductI}u*v=(\Delta^{J,L}_I)^*(u\otimes v) \in H^{p+q}(\widehat{Z}(K_I)).\end{equation}
In \cite{bbcg3} the ring structure  of $H^*(Z(K,(\underline{X},\underline{A}))$ is shown to be induced by  the $\ast$-product.

\

The special case of \noindent $(\underline{X},\underline{A})=(D^1,S^0)$ is particularly important.

\

The splitting of  Theorem, \ref{thm:decompositions.for.general.moment.angle.complexes}, and Theorem \ref{thm:null.A} imply there  are homotopy equivalences

\

$$\Sigma Z(K;(D^1,S^0)) \to \underset{I \subset [m]}{\bigvee} \Sigma \widehat{Z}(K_I; (D^1,S^0))$$
and
\
$$\widehat{Z}(K_I;(D^1,S^0))\stackrel{\simeq}{\to}|K_I|*(\widehat{S}^0)^I\simeq \Sigma |K_I|.
$$

\

With  $I = J \cup L$ a pairing

\begin{equation}\label{equa:starD1} H^*(\Sigma |K_J|) \otimes H^*(\Sigma |K_L|) \to H^*(\Sigma|K_I|)\end{equation}
is induced by
the  partial diagonals  map $\Delta^{J,L}_I.$
$$ \Sigma |K_I|  \simeq  \widehat{Z}(K_I;(D^1,S^0))  \overset{\Delta_I^{J,L}}{\to} \widehat{Z}(K_J;(D^1,S^0)) \wedge \widehat{Z}(K_L ;(D^1,S^0)) \simeq \Sigma|K_J| \wedge \Sigma |K_L|.$$

\

\begin{thm}\label{ringstar} The product
	
	\
	
	$$ \Big[n_{\alpha} \otimes a_1 \otimes \cdots \otimes a_m \Big]
\cup \Big[n_{\gamma} \otimes g_1 \otimes \cdots \otimes g_m \Big] \in \hz$$is given by the $\ast-$product of $n_{\alpha}$ and $n_{\beta}$  composed with an inclusion map described in Lemma \ref{lem:diagfactor} and a coordinate wise product defined as follows:
	
\begin{enumerate} 	
\item  If $a_i, g_i \in H^*(X_i)$ the product in the $i$-th coordinate is the product in  $H^*(X_i).$
\item If $a_i, g_i \in H^*(A_i) $ the product is induced by the product in $H^*(A_i)$.

\item if $a_i \in E_i, g_i \in C_i$ or $g_i \in E_i, a_i \in C_i$
the product is zero.

\end{enumerate}

\end{thm}

The rest of this section is devoted to proving Theorem \ref{ringstar}.

\

A description of   the diagonal map $$\zk \to \zk \times \zk$$and the partial diagonal maps
$$\widehat{Z}(K_I; \xa_I) \to \widehat{Z}(K_J;\xa_J)\wedge \widehat{Z}(K_L; \xa_L)$$ are now given.    The description uses the following lemma  where the notation 

\

$$ [ \xa^{J.L}_I ]_i =\begin{cases}
(X_i,A_i) & \mbox{ if } i \in I -J \cap L\\
(X_i \wedge X_i, A_i \wedge A_i) & \mbox{ if } i \in J \cap L
\end{cases}
$$
is used in part (c).

\begin{lem}\label{Zdiagonal}
\begin{enumerate}
\item[(a)] Suppose $A_m\subset X_m,  B_m \subset Y_m$  then there is a natural map  $$Sh: Z(K; (\underline{X  \times Y}, \underline{ A \times B})) \to Z(K; (\underline{X},\underline{A})) \times Z(K; (\underline{Y},\underline{B}))$$  where $$ (\underline{X  \times Y}, \underline{ A \times B})) = \{(X_m \times Y_m, A_m \times B_m)\}.$$

\

\item[(b)]The diagonal maps of pairs, $ (X_m,A_m) \to (X_m \times X_m, A_m \times A_m)$ defines a map \newline $$\widetilde{\Delta}: (\underline{X},\underline{A}) \to ( \underline{X \times X}, \underline{A \times A})$$ which induces a map of spectral sequences.

$$E_r(Z(K; (\underline{X \times X}, \underline{A \times A}))) \to E_r(Z(K; \xa))$$

The diagonal map  is the composite $$\Delta: \zk \overset{\widetilde{\Delta}}{\to} Z(K; (\underline{X \times X}, \underline{A \times A}) ) \overset{Sh}{\to}  \zk \times \zk.$$

\

\item[(c)] The partial diagonal map of pairs $\widehat{\Delta}^{J,L}_I: \xa \to \xa^{J.L}_I $

$$ \widehat{\Delta}^{J,L}_I = \begin{cases}
\mbox{ the identity } & \mbox{ if } i \in I - J \cap L\\
\mbox{ the reduced diagonal } & \mbox{ if } i \in J \cap L.
\end{cases}$$ induces a map of spectral sequences.
$$E_r(\widehat{Z}(K_I; \xa^{J,L}_I)) \to E_r(\widehat{Z}(K_I; \xa_I))$$

$\Delta^{J,L}_I$  is the composite
$$ \widehat{Z}(K_I; \xa_I) \overset{ \widehat{\Delta}^{J,L}_I }{\to} \widehat{Z}(K_I; \xa^{J,L}_I) \overset{Sh}{\to} \widehat{Z}(K_J; \xa_J) \wedge \widehat{Z}(K_L; \xa_L).$$

\end{enumerate}
\end{lem}

\begin{proof}   (a):  From the definition of the polyhedral product functor, Definition \ref{defin:gmac}, \newline $Z(K; (\underline{X  \times Y}, \underline{ A \times B}))$  is a colimit of spaces,  $D(\sigma,(\underline{X  \times Y}, \underline{ A \times B}) )$.  By shuffling the factors,  there is a  map  $D(\sigma,(\underline{X  \times Y}, \underline{ A \times B}) )\to D(\sigma, (\underline{X}, \underline{A})) \times D(\sigma, (\underline{Y}, \underline{B}))$.  The maps are compatible with the maps into  the spaces of the  colimit defining  $Z(K; (\underline{X},\underline{A})) \times Z(K; (\underline{Y},\underline{B}))$,  proving (a).

 (b): The diagonal
 $$\Delta: \zk \to  \zk \times \zk$$
 is  induced at the level of $D(\sigma)$ by $\widetilde{\Delta}$ followed by a shuffle.  It follows from (a) that $\Delta$ factors as indicated.

 (c):  Is similar.

  \end{proof}

\

Since the product in $H^*(\zk)$ is determined by the $\ast-$product, as in (\ref{starproductI}) it suffices to compute the map induced by $\Delta^{J,L}_I$.  It follows from Lemma \ref{Zdiagonal}  that $\Delta^{J,L}_I$  decomposes as the composition of the  map induced by the shuffle followed by the map induced by $\widehat{\Delta}^{J,L}_I $.  The shuffle map may be computed using the fact that the spectral sequence, and hence $H^*(\widehat{Z}(K; \xa)$ is a functor of  the strong $H^*-$cohomology type (Corollary \ref{cor:specSeqstrong isomorphism}) and the decomposition in
Definition \ref{def:bas}.  The map induced by $\widehat{\Delta}^{J,L}_I$  is computed using the naturally  discussed in remark \ref{rem:naturality}.

\

  We first describe  $(\widehat{\Delta}^{J,L}_I)^*$. To this end the decomposition of Definition \ref{def:bas} is described for the pair $(\underline{X \wedge X}, \underline{A \wedge A})$.
 
 \

 There is the long exact sequence
 $$\overset{\delta}{\leftarrow} \widetilde{H}^*(A_i \wedge A_i) \overset{i}{\leftarrow} \widetilde{H}^*(X_i
 \wedge X_i) \leftarrow \widetilde{H}^*(X \wedge X/A \wedge A) \leftarrow $$
 The image, kernel and cokernel of Definition \ref{def:bas} for the above exact sequence will be denoted $\widehat{B}_i, \widehat{C}_i$ and $\widehat{E}_i$ respectively.  In terms of $B_i,C_i,E_i$ associated to the pair $(X_i,A_i)$

 \begin{equation}\label{hats}
 \begin{array}{lcl}
 \widehat{B}_i & =& B_i \otimes B_i\\
 \widehat{C}_i & =& \Big(C_i \otimes C_i\Big) \oplus \Big( B_i \otimes C_i\Big)  \oplus \Big(C_i \otimes B_i\Big)\\
 \widehat{E}_i &=& \Big(E_i \otimes E_i\Big) \oplus \Big(E_i \otimes B_i \Big)\oplus \Big(B_i \otimes E_i\Big)
 \end{array}
 \end{equation}

 The diagonal induces the product in $H^*(X_i)$ and $H^*(A_i)$

 \
 
  \begin{equation}\label{productmap}
  \begin{array}{lcl}
 \widehat{B}_i & \to &H^*(X_i) = B_i \oplus C_i\\
 \widehat{C}_i & \to & C_i \\
 \widehat{E}_i &\to & H^*(A_i) = E_i \oplus B_i
 \end{array}
 \end{equation}

 The map of long exact sequences  induced by the diagonal
 
 \begin{equation}
 \begin{array}{ccccccccccc}
 \cdots \overset{\delta}{\leftarrow}&h^*(A_i) &\leftarrow&h^*(X_i)&\leftarrow & \widetilde{h}^*(X_i/A_i) & \leftarrow \cdots\\
 &\uparrow{\Delta_i^*} && \uparrow{\Delta^*_i} && \uparrow\\
 \cdots \overset{\delta}{\leftarrow} &h^*(A_i \wedge A_i)&\leftarrow&h^*(X_i\wedge X_i) &\leftarrow & \widetilde{h}^*(X_i\wedge X_i/A_i \wedge A_i) &\leftarrow \cdots
 \end{array}
 \end{equation}

  is an example of the more general situation described at the end of Section \ref {revisited}.
  
  \
  
  The map	

  \

   $$\widehat{\Delta}^{J,L}_I: \widehat{Z}(K_I; \xa_I) \to \widehat{Z}(K_I; \xa^{J,L}_I)$$is induced by a map of pairs $\xa_I \to  \xa^{J,L}_I$  which at  $E_{\infty}$    is   the product described in (\ref{productmap}).

\

The details follow.  We first have to adjust the indexing sets in Theorem \ref{Ps}.  The vertex set of the  simplicial complex $K_I$ is $I$ not $[m]$.  Since $I$ now denotes  the vertex set we cannot use it in the notation for the factor $Y^{I,\sigma}$.  We replace   $I$ in Theorem \ref{Ps}  with $F$.  With these modifications  

\

$$ H^*(\widehat{Z}(K_I; \xa^{J,L}_I)) $$is a sum of groups
$$\underset{F \subset I, \sigma \in K_I}{\bigoplus} \overline{E}^{I-F} \otimes \widetilde{H}^*(\Sigma |N(F,\sigma)|) \otimes \overline{Y}^{F,\sigma}$$

\

\noindent where a factor $\overline{E}_i$ of $\overline{E}^{I-F}  $ is $\widehat{E}_i$ if $i \in J \cap L$ and  $E_i$ otherwise.  The factors of $ \overline{Y}^{F,\sigma}$ are defined analogously.  Specifically a factor $\overline{C}_i$ of $\overline{Y}^{F,\sigma}$ is $\widehat{C}_i$ if $i \in J \cap L$ and $C_i$ otherwise.  A factor  $\overline{B}_i$ of $\overline{Y}^{F,\sigma}$ is $\widehat{B}_i$ if $i \in J \cap L$ and is $B_i$ otherwise.

\

Similarly 
$$ H^*(\widehat{Z}(K_I; \xa_I )) $$
 is a sum
 
 \

$$ \underset{F^{\prime} \subset I , \tau \in K_I}{\bigoplus} E^{I-F^{\prime}} \otimes \widetilde{H}^*(\Sigma|N( F^{\prime} ,\tau)|) \otimes Y^{F^{\prime},\tau}.$$

\

The map   $$\widehat{\Delta}^{J,L}_I: \widehat{Z}(K_I; \xa_I) \to \widehat{Z}(K_I; \xa^{J,L}_I)$$

\

 induces a map in cohomology which 
 restricted to each summand is  a map
 
 \
 
$$\overline{E}^{I-F} \otimes \widetilde{H}^*(\Sigma |N(F,\sigma)|) \otimes \overline{Y}^{F,\sigma} \to \underset{F^{\prime} \subset I, \tau \in K_I}{\bigoplus} E^{I-F^{\prime}} \otimes \widetilde{H}^*(\Sigma|N(F^{\prime},\tau)|) \otimes Y^{F^{\prime},\tau}.    $$

 \
 
 This map is computed using the naturality discussion at the end of Section  \ref{revisited}  with maps on $\widehat{E}_i, \widehat{B}_i,$  and  $\widehat{C}_i$ given by (\ref{productmap}).  Specifically the map induced by the product on   $\widehat{E}_i$ may have summands in    $B_i$ thus enlarging   $F$ to a larger indexing set, $F^{\prime}$.   Also some of the terms  $\widehat{B}_i$ which appear in $\overline{Y}^{F,\sigma}$  map via the product to summands with a factor of $C_i$ enlarging the simplex $\sigma$ to $\tau$.   Thus we have proven lemma \ref{lem:diagfactor} below. The sum is over all  $F^{\prime} \subset I$ and $\tau$.  

 \begin{lem}\label{lem:diagfactor}
$$(\widehat{\Delta}^{J,L}_I)^*: \overline{E}^{I-F} \otimes \widetilde{H}^*(\Sigma |N(F,\sigma)|) \otimes \overline{Y}^{F,\sigma} \to  \underset{F^{\prime},\tau}{\bigoplus} E^{I-F^{\prime} } \otimes \widetilde{H}^*(\Sigma|N(F^{\prime},\tau)|) \otimes Y^{F^{\prime},\tau}, $$
where $ F^{\prime} \supset F, \tau \supset  \sigma$, the  product,  (\ref{productmap}), induces the maps on the factors in $\overline{E}$ and $\overline{Y}$ and \newline
 $ \widetilde{H}^*(\Sigma |N(F,\sigma)|) \to  \widetilde{H}^*(\Sigma |N(F^{\prime},\tau)|)$ is induced by the inclusions $$  N(F^{\prime},\tau) \to  N(F,\sigma).$$
\hspace {6  in}   $ \square$

\end{lem}

\

Intuitively the diagonal induces the product on the coordinates.  Because of the mixing of $E's, B's$ and $C's$ the cohomology of the links map to the cohomology of the resulting sub-links .  

\

Note that (\ref{hats}) implies  neither $E_i \otimes C_i$ nor $C_i \otimes E_i$ appear in $H^*(\widehat{Z}(K_I; \xa^{J,L}_I))$.  This implies the product of classes in $\zk $  involving $C_i$ and $E_i$ must be zero.

\

An important special case is that of wedge decomposable spaces.

\

\begin{defin}
	A collection of spaces, $\xa$ is wedge decomposable  if for all $i$ 
	$$X_i = B_i \vee C_i$$ and
	$$ A_i = B_i \vee E_i$$ where $$E_i \to B_i \vee C_i$$ is null homotopic.
	
	\end{defin}
	
	In this case there is none of the mixing of the products complicating the map in Lemma \ref{lem:diagfactor}.

\

\begin{cor}
	If $\xa$ is wedge decomposable then the product of
	
		$$ \Big[n_{\alpha} \otimes a_1 \otimes \cdots \otimes a_m \Big]
		\cup \Big[n_{\gamma} \otimes g_1 \otimes \cdots \otimes g_m \Big] \in \hz$$is given by the $\ast-$product of $n_{\alpha}$ and $n_{\beta}$  and a coordinate wise product defined as follows:
		
		\begin{enumerate} 	
			\item  If $a_i, g_i \in H^*(B_i)$ the product in the $i$-th coordinate is the product in  $H^*(B_i).$
			\item If $a_i, g_i \in H^*(C_i) $ the product is induced by the product in $H^*(C_i)$. 
			\item If $a_i, g_i \in H^*(E_i) $ the product is induced by the product in $H^*(E_i)$.
			\item The product is zero otherwise.
			\end{enumerate}
\end{cor}
\begin{comment}
\begin{remark}\label{rem:recapitulation}   We recapitulate the various cases that appear in Lemma \ref{lem:diagfactor}.

\

\begin{itemize}
\item If each $\widehat{E}_i, \widehat{C}_i, \widehat{B}_i$ maps to $E_i, C_i$  and $B_i$ then 
$(\widehat{\Delta}^{J,L}_I)^*$ is simply the product in each factor and the identity on $H^*(\Sigma|N(F,\sigma)|).$

\
\item If  a factor $\widehat{B}_i$ of $Y^{F,\sigma}$ maps to $C_i$ then the indexing set $F$ does not change.  However for the product to produce a non zero term  $\sigma \cup \{i\} = \sigma^{\prime}$ must be a simplex in $K$.  If this is the case  there may be  a multiplicative extension in the spectral sequence.

In fact all differentials and extensions are taking place in the spectral sequences of the links,  $H^*(Z(N;(D^1,S^0))$.

\

\item If a factor $\widehat{E}_i$ maps to $B_i$ then the indexing set $F$ is replaced by $F^{\prime}=F \cup \{i\}$ and $\sigma$ does not change.  There are two sub cases depending on the support of $\widehat{E}_i$  (see Remark \ref{remark:EandW}):

\
\begin{enumerate}
\item  $N(F,\sigma) = N(F^{\prime}, \sigma)$.  Then $\widehat{E}_i$ is supported on $H^*(A_i \wedge A_i)$.

\

\item $N(F, \sigma) \supset N(F^{\prime},\sigma)$.  Then $\widehat{E}_i$ is supported on $H^*(X_i \wedge X_i/A_i \wedge A_i)$.

\end{enumerate}

\end{itemize}

\end{remark}

\end{comment}

\begin{exm}\label{exm:productOnD1S0} We illustrate Lemma  \ref{lem:diagfactor} for $H^*(\widehat{Z}(K; (D^1,S^0)))$.

$E_i$ is generated by a zero dimensional class, $t_0$ for all $i$.  The product $E_i \otimes E_i \to E_i$ is an isomorphism ( $t_0 \otimes t_0 \mapsto t_0$).  In particular $F=F^{\prime}$.

\
In this case the pair  $(D^1 \wedge D^1, S^0 \wedge S^0) $ is homotopy equivalent to $(D^1,S^0)$ and
$\widehat{\Delta}^{J,L}_I$ is a homotopy equivalence.

As a consequece the $\ast-$product, (\ref{equa:starD1})

$$ H^*(\Sigma |K_J|) \otimes H^*(\Sigma |K_L|) \to H^*(\Sigma|K_I|)$$

\
is completely determined by the map induce by the shuffle
\end{exm}

$$Sh^*: \widetilde{H}^*(\widehat{Z}(K_J; (D^1,S^0)_J))  \otimes \widetilde{H}^*( \widehat{Z}(K_L; (D^1,S^0)_L)) \to \widetilde{H}^*( \widehat{Z}(K_I; (D^1,S^0)^{J,L}_I) ).$$

\

In order to describe the suffle map for general $\xa$ satisfying the freeness condition it is convenient to give yet another description of the spectral sequence.  Motivated by the proof of Theorem \ref{Ps} we  introduce variables, $t_i$ and $s_i$ into the $E_1$ term of the spectral sequence.  The degree of $t_i = 0$ and the degree of $s_i=1$. 

\

For $K$ a simplicial complex with vertex set $[m]$ recall the $E_1$ term is a sum of groups

$$ \widetilde{H}^*(X_i/A_i)^{\tau} \otimes \widetilde{H}^*(\widehat{A}_i)^{\tau^c}$$ ($\tau$ a simplex of $K$).
Which in turn  is a sum of groups

\

\begin{equation}\label{sumofGroups}C^{\sigma}\otimes W^{\sigma^{\prime}} \otimes E^P \otimes B^Q \end{equation}

\

Where $\sigma$ and $\sigma^{\prime }$ are simplices of  $K$ whose union is is a simplex $\tau \in K$  and $P \cup Q = \tau^c$.   In the notation of Theorem \ref{Ps} $,  Q \cup |\sigma| = I$, $P \cup |\sigma^{\prime}| = [m]-I$.  The filtration of this summand is the weight of the simplex $\tau$ in the left lexicographical ordering of the simplices.

\

Replace $E_i$ with $t_i E_i$ and $W_i$ with $s_i E_i$ in (\ref{sumofGroups}) and arrange the sum as follows

\

\begin{equation}\label{sumGroupsII}
E_1=\underset{I\cup J =[m], I \cap J =\emptyset, \sigma \in K_I}{\underset{I,J, \sigma}{\bigoplus} }\Big( \underset{P \cup |\sigma^{\prime}|=J, \sigma \cup \sigma^{\prime} \in K}{\underset{P, \sigma^{\prime}}{\bigoplus}} E^J \otimes Y^{I,\sigma} t^P s^{\sigma^{\prime} }\Big)
\end{equation}

\

Define a differential on $E_1$ by $\delta t_i = s_i$, $\delta s_i=0$.  The proof of Theorem \ref{Ps} implies the following proposition.

\

\begin{prop} \label{prop:e1withts}

\begin{enumerate}
\item The spectral sequence
$$E_r(\widehat{Z}(K;\xa) \Rightarrow H^*(\widehat{Z}(K; \xa))$$
is isomorphic to the spectral sequence with $E_1$ term
$$\underset{I\cup J =[m], I \cap J =\emptyset, \sigma \in K_I}{\underset{I,J, \sigma}{\bigoplus} }\Big( \underset{P \cup |\sigma^{\prime}|=J, \sigma \cup \sigma^{\prime} \in K}{\underset{P, \sigma^{\prime}}{\bigoplus}} E^J \otimes Y^{I,\sigma} t^P s^{\sigma^{\prime} }\Big)$$ and  differential $\delta t_i = s_i$.

\

\item Let $N(I,\sigma)$ be the simplicial complex defined in Definition \ref{def:N(Isigma)}.  Then
$$ \underset{P \cup |\sigma^{\prime}|=J, \sigma \cup \sigma^{\prime} \in K}{\underset{P, \sigma^{\prime}}{\bigoplus}} E^J \otimes Y^{I,\sigma} t^P s^{\sigma^{\prime} }$$
is isomorphic to
$$\Big(E^J \otimes Y^{I,\sigma}\Big) \otimes E_1(\widehat{Z}(N(I,\sigma) ; (D^1, S^0))$$as differential groups.

\end{enumerate}
\end{prop}

 It will be  convenient to write $N_J$ for $ |N(I,\sigma|)$. So the summands of  $H^*(\zkh)$ have the form
$$E^J \otimes H^*(\Sigma  |N_J|) \otimes Y^{I,\sigma}.$$

\

The shuffle map is induced by
$$\widehat{D}(\sigma_I) \to \widehat{D}(\sigma_J) \wedge \widehat{D}(\sigma_L)$$
where $\sigma_J=\sigma_I \cap J, \sigma_L = \sigma_I \cap L$.
This map induces the shuffle on the $E, C$ and $B$ terms and  induces the shuffle on the  $t_i,p_i$ terms.
Specifically the shuffle map induces a map

$$Sh^*:  \big(E^{J_1} \otimes Y^{I_1,\sigma_1} t^{P_1} s^{\sigma_1^{\prime} }\big) \otimes
\big(E^{J_2} \otimes Y^{I_2,\sigma_2} t^{P_2} s^{\sigma_2^{\prime} } \big)\to \bigoplus  \overline{E}^J \otimes \overline{Y}^{I,\sigma} t^P s^{\sigma^{\prime} }$$
with $P_1 \cup |\sigma_1^{\prime}| = J_1, P_2 \cup |\sigma_2^{\prime}| = J_2.$  $\overline{E}$ and $\overline{Y}$ are define in the paragraph before Lemma \ref{lem:diagfactor} and $P \cup |\sigma^{\prime}|=J$.

\

From (\ref{hats}) it follows that $Sh^*$ is zero if there is a coordinate, $i$ with $E_i$ a factor in    $E^{J_1} \otimes Y^{I_1,\sigma_1} t^{P_1} s^{\sigma_1^{\prime} }$  and a factor $C_i$ in $E^{J_2} \otimes Y^{I_2,\sigma_2} t^{P_2} s^{\sigma_2^{\prime} } $ (or visa-verse).  Hence there is no loss of generality to assume 

\begin{equation}\label{condition}
J_1 \cap |\sigma_2| = \emptyset = J_2 \cap |\sigma_1|\end{equation}

\

Next notice that  $Sh^*$ takes any coordinate with a factor of $E_i$ to $\overline{E}_i$, any  coordinate involving $C_i$ to $\overline{C}_i$ and coordinates with both factors in $B_i$ to $\overline{B}_i$.

\
Hence the image of $Sh^*$ takes values  in the summand

$$ \overline{E}^{J_1 \cup J_2} \otimes \overline{Y}^{I_1\cap I_2, \sigma_1 \cup \sigma_2}t^P s^{\sigma^{\prime}}$$
where $P \cup |\sigma^{\prime}|=J_1 \cup J_2$ and $\sigma^{\prime} \cup \sigma_1 \cup \sigma_2 $ is a simplex in $K$. In particular  $Sh^*$ is zero on factors  were  $\sigma_1 \cup \sigma_2$ is not a simplex in $K$.   It is a consequence of  (\ref{condition}) that $|\sigma_1 \cup \sigma_2| \subset I_1 \cap I_2$.

Next observe that the shuffle map on $t_i$ and $s_i$ is exactly the map induced by the shuffle map for
$$H^*(\Sigma |N_{J_1}|) \otimes H^*(\Sigma |N_{J_2}|) \to H^*(\Sigma |N_{J=J_1 \cup J_2}|)$$   Which by example \ref{exm:productOnD1S0},  is the $\ast-$product.  Using work of Cai \cite{LC} we will give a chain level formula for the $\ast-$product in section \ref{section7}.

\

 Thus $Sh^*$ is the shuffle map on the $E, B$ and $C$ factors and the $\ast$-product on the cohomology of the links.   Lemma \ref{lem:diagfactor} and the computation of $Sh^*$ prove Theorem \ref{ringstar}.

\

We have shown that the map induced by the shuffle only depends on the ring structure of the cohomology of the subcomplexes of $K$ and the strong homology type of $H^*(X_i)$ and $H^*(A_i)$.  The contributions of the  ring structure of $H^*(X_i)$ and $H^*(A_i)$ to $H^*(\zkh)$ appear in the diagonal map, $\widehat{\Delta}^{J,L}_I$.  So decomposing the partial diagonal into the diagonal composed with the shuffle separates the combinatorial contribution to the ring structure of $\zk$ from the cup product structure of the cohomology of $\xa$.

 \begin{remark}

 \begin{enumerate}

\item Q.~Zheng, \cite{zheng},  also describes a product in $\hz$.

  \
  
  \item
 Proposition  \ref{prop:e1withts} generalizes to a multiplicative cohomology theory, $h^*$, satisfying the flatness condition.

 \end{enumerate}

 \end{remark}

\vspace {0.5 in}

\section{{\bf  The Cohomology of $Z(K;(\underline{D^1}, \underline{S^0}))$}}\label{section7}

Assuming suitable freeness conditions the results of Section \ref{sec: products} determine the cohomology ring $H^*(\zk) $ in terms of the rings $H^*(X_i)$,  $H^*(A_i)$ and the star product on the cohomology of the links. The goal of this section is to complete this description by describing the star product.  

\

Recall that the product 

$$H^*(\zk) \otimes H^*(\zk) \to H^*(\zk)$$

\

 is the composite of the map induced by 
the shuffle 

\

$$Sh^*:H^*(\zk) \otimes H^*(\zk) \to H^*(Z(K;(\underline {X \times X}, \underline{A\times A})))$$

\

and the map induced by the diagonal
 \
  $$\Delta^*:  H^*(Z(K;(\underline {X \times X}, \underline{A\times A}))) \to H^*(\zk).$$

\

On the summands of Theorem \ref{Ps} it was shown in section \ref{sec: products} that the shuffle has the form

\begin{multline} Sh^*:  \big(E^{J_1} \otimes H^*(\Sigma|N_{J_1}|)) \otimes  Y^{I_1,\sigma_1}\big) \otimes
\big(E^{J_2} \otimes H^*(\Sigma |N_{J_2}|) \otimes Y^{I_2,\sigma_2}  \big)  \\  \to 
\overline{E}^{J_1 \cup J_2} \otimes H^*(\Sigma |N_{J_1\cup J_2}|) \otimes  \overline{Y}^{I_1 \cap I_2,\sigma_1 \cup \sigma_2} 
  \end{multline}

\

\noindent where the pairing 

\begin{equation} \label{shuffelLinks} H^*(\Sigma|N_{J_1}|)) \otimes H^*(\Sigma |N_{J_2}|) \to  H^*(\Sigma |N_{J_1\cup J_2}|)\end{equation} is induced by the $\ast-$product.  

\

In section \ref{sec: products} it was also shown that the diagonal has  the form
 
 \
 $$ \overline{E}^{J_1 \cup J_2} \otimes H^*(\Sigma |N_{J_1\cup J_2}|) \otimes  \overline{Y}^{I_1 \cap I_2,\sigma_1 \cup \sigma_2} \to E^{J^{\prime}} \otimes H^*(\Sigma |N_{J^{\prime}}|) \otimes Y^{F^{\prime}, \tau}$$ 
 
 \
 
 \noindent with the map on the link 
 \begin{equation}\label{diagLink}H^*(\Sigma |N_{J_1\cup J_2}|) \to  H^*(\Sigma |N_{J^{\prime}}|) \end{equation} induced by an inclusion $\iota: N_{J^{\prime}} \to N_{J_1 \cup J_2}$.
 
 \

 We describe the $\ast-$product and inclusion map on the links by  constructing a filtered chain complex for   $H^*(Z(K; (D^1,S^0)))$.   The spectral sequence associated to this filtered complex is the spectral sequence  $E_r(Z(K; (D^1,S^0)))$.

\

For $I \subset [m]$ recall that there are classes, $s_i$ of degree $1$ and $t_i$ of degree $0$ such that   the $E_1$ page of the spectral sequence for $\widehat{Z}(K_I;(D^1,S^0))$ is generated by
$$ y_{\sigma} =y_1\otimes \cdots  \otimes y_m$$
where $\sigma$ is a simplex in $K_I$ and 
$$y_i=
\begin{cases}

s_i & \mbox{ if } i \in \sigma \\
 t_i & \mbox{ if } i \notin \sigma \mbox{ and } i \in I\\
 1 & \mbox{ if } i \notin I.
\end{cases}$$
Where $1$  is the multiplicative unit.

\

Define a cochain complex,  $C_{K_I}$ freely generated over $\mathbb{Z}$   by $\{y_{\sigma}| \sigma \in K_I\} $ with differential
$$ d(y_{\sigma} )= \underset{\tau}{\Sigma}  \mbox{  } (-1)^{n(\tau)} y_{\tau}$$
where, for $\sigma$ an $n-$simplex, the sum is over all $n+1$ simplices  $\tau$ such that $\sigma \subset \tau \in K_I$.  The integer 
$n(\tau)$ is defined by the usual graded sign convention for a derivation.  i.e. there is a differential  $\delta$ acting on each factor by $\delta(t_i) = s_i$, $\delta(s_i)=\delta(1) =0$. $d$ is defined on $y_{\sigma}$ by extending $\delta $ to $y_{\sigma}$ by the graded  Leibniz rule.  

\begin{comment}
The sign is  $\delta $y_{\tau}$ is obtained from $y_{\sigma}$  by passing $\delta$ past $s$'s and $t$'s until it gets to the coordinate $i$ where $s_i \in \tau \smallsetminus \sigma$ .  A factor of $(-1)$ is added each time $\delta $ passes by an $s_i$.   

\end{comment}

Define $C_K$ by
\begin{equation} \label{def:chaincomplex} C_K = \underset{I\subset [m]}{\bigoplus} C_{K_I}.\end{equation}  The following is proven in \cite{LC}.

\begin{prop}\label{prop:chaincomplex}  There is an isomorphism of groups 

  $$H^*(C_K) = H^*(Z(K; (D^1,S^0)))$$
  
\end{prop}  

\begin{proof} $C_K$ is the dual of the standard chain complex for $\underset{I}{\bigvee} \Sigma |K_I|$.  The proposition now follows from Theorem \ref{thm:null.A}.
\end{proof}

\

The left lexicographical ordering of the simplices of $K$ induce a filtration on  \newline $C_K= \underset{I}{\bigoplus} \mathbb{Z}\{y_{\sigma}| \sigma \in K_I\} .$  The spectral sequence associated to this filtration is easily seen to be the spectral sequence $E_r(Z(K;(D^1,S^0)))$.  

\

In \cite{LC} Cai gives a non-commutative product on the chain complex (\ref{def:chaincomplex}) which induces the cup product in $H^*(Z(K; (D^1,S^0)))$.  This product specializes to the $\ast-$product of Theorem \ref{ringstar}.

 Following \cite{LC} we define a non commutative product on $C_K$ by extending the following  product on the classes $t_i$ and $s_i$:
\begin{equation}\label{caiproduct}
t_i t_i = t_i,\quad
	t_i s_i = 0,\quad
	s_i t_i = s_i,\quad
	s_i s_i = 0
\end{equation}

\

\noindent to $C_K$ by the graded shuffle.

\

The product (\ref{caiproduct}) induces the  $\ast-$product, (\ref{shuffelLinks}), on the cochain complexes for $N_{J_1} $ and $N_{J_2}$ as follows.  We suppose there are simplices  $\alpha \in N_{J_1}$ and $\beta \in N_{J_2}$.  There are the generators $y_{\alpha}$ and $y_{\beta}$ of the cochains of $N_{J_1}$ and $N_{J_2}$ respectively.   The graded shuffle product of $y_{\alpha}$ and $y_{\beta}$ followed by the coordinate wise product defined in (\ref{caiproduct}) determine a signed monomial, $y_{\gamma}$ in the $t_i$'s and $s_i$'s.  Here   $\gamma = \{i|s_i \mbox{ appears in the } i-th \mbox{ coordinate of the monomial} \}$.  
The  $\ast-$product, (\ref{shuffelLinks}) on cochains  is given by 

$$y_{\alpha} \otimes y_{\beta} \mapsto \pm y_{\gamma} \mapsto \begin{cases}
\pm y_{\gamma} & \mbox{ if } \gamma \cup \sigma_1 \cup \sigma_2 \in K \\
0 & \mbox{ otherwise } 
\end{cases}$$

\

The map (\ref{diagLink}) is induced by the map of cochains dual to the inclusion.  Namely if $y_{\gamma}$ is a generator of the cochains of $N_{J_1 \cup J_2}$ then $\iota^*(y_{\gamma})$ is given by

\

$$ y_{\gamma} \mapsto 
\begin{cases}
y_{\gamma} & \mbox{ if } \gamma \in N_{J^{\prime}} \\
	0 & \mbox{ otherwise }
	\end{cases} $$

	\vspace {0.5 in}

{\bf Acknowledgments}. The first author was supported in part by a Rider University Summer Research
Fellowship and grant number 210386 from the Simons Foundation; the third author was supported partially by
DARPA grant number 2006-06918-01.

%\noindent{\bf Acknowledgements}.  The first author was supported in part by a Rider University Summer Research Fellowship
%and grant number 210386 from the Simons Foundation.
 %The third
%author was partially supported by DARPA grant number 2006-06918-01.
%The authors would all like to thank the IAS for their hospitality and fertile environment
%while this paper was in preparation.

\bibliographystyle{amsalpha}

\end{document}